\documentclass[12pt]{article}
\usepackage[dvips]{epsfig}
\usepackage{amsfonts,amssymb}
\usepackage{times}
\usepackage{graphicx}

\makeatletter\@addtoreset {equation}{section}\makeatother

\newtheorem{theorem}{Theorem}[section]
\newtheorem{lemma}[theorem]{Lemma}
\newtheorem{definition}[theorem]{Definition}

\newenvironment{proof}{
    \noindent {\it Proof.}}{\hfill$\Box$
}

\def\eps{\varepsilon}

\def\RR{{\mathbb R}}
\def\ZZ{{\mathbb Z}}

\begin{document}

\title {Convergence to equilibrium for
a thin film equation\\ on a cylindrical surface}

\author{Almut Burchard ~~~~~~~~ Marina Chugunova\\
University of Toronto\\
{\tt \{almut,chugunom\}@math.utoronto.ca}\\
[0.5cm]
Benjamin K. Stephens\\
University of Washington\\
{\tt benstph@math.washington.edu}}

\date{November 30, 2011}

\bigskip\bigskip

\maketitle

\begin{abstract} The degenerate parabolic equation
$u_t+\partial_x [u^3\,(u_{xxx} + u_x-\sin x)]=0$ 
models the evolution of a thin 
liquid film on a stationary horizontal cylinder. 
It is shown here that for each mass there 
is a unique steady state, given by a droplet hanging from the bottom 
of the cylinder that meets the dry region with zero contact angle. 
The droplet minimizes the associated energy functional 
and attracts all strong solutions that 
satisfy certain energy and entropy inequalities,
including all positive solutions.
The distance of solutions from the steady state cannot decay faster 
than a power law. 



\end{abstract}

\section{Introduction and description of the results}

Degenerate fourth order parabolic equations of the form
\begin{equation}
\label{eq:thin-film}
u_t + \nabla \cdot (u^n \nabla \Delta u) + \mbox{lower order terms}=0
\end{equation}
are used to model the evolution of thin liquid films on solid surfaces. 
Here, $u(x,t)$ describes the thickness of the fluid at time $t$ at 
the point $x$, the fourth derivative term models the surface tension,
and the exponent $n>0$ is determined by the boundary condition between 
the liquid and the solid.  The equations were derived from the 
underlying free-boundary-value problem for the Navier-Stokes equation
by the lubrication approximation, which is valid if the film is 
relatively thin.

\begin{figure}[ht] 
\begin{center}
\includegraphics[height=4cm] {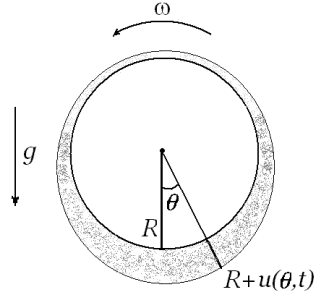}\\
\end{center}
\caption{\small Thin liquid film on the outer surface of a horizontal cylinder.}
\end{figure}

An interesting example is Pukhnachev's model for a thin liquid film on 
a rotating horizontal cylinder~\cite{Pukh1}, as shown in Figure~1. 
The model equation is 
\begin{equation} \label{eq:Pukh}
u_t  + \partial_\theta
\left[ \frac{1}{3} \frac{\sigma}{\rho \nu R^4 } u^3
\left( u_{\theta\theta\theta}+u_\theta \right) 
- \frac{1}{3} \frac{g}{\nu R} u^3 \sin \theta  + \omega u
\right] =0\,.
\end{equation} 
Here, $\theta \in \RR/(2\pi\ZZ)$ denotes the angle measured from the 
bottom of the cylinder, the film is assumed to be uniform in the axial 
direction, and inertia is ignored.  The power $n=3$ indicates a 
{\em no-slip} boundary condition between the liquid and the solid. 
The other physical parameters of the system are the surface 
tension ($\sigma$), the viscosity ($\nu$), the acceleration of 
gravity ($g$), the density of the fluid ($\rho$), the radius of the 
cylinder ($R$), and the rotation speed ($\omega$).  Note that the 
surface tension appears with an additional lower-order term that 
corrects for the curvature of the cylinder.  
Eq.~(\ref{eq:Pukh}) refines an earlier model
of Moffatt~\cite{Moffatt} that neglects surface tension.
Numerical and asymptotical analysis of 
Pukhnachev's equation along with numerous open questions 
can be found in \cite{Benilov3,Benilov1,Kar},
and linearizations about steady states are examined 
analytically and numerically in~\cite{ChKP09,John}. 

We will show that the long-time behaviour of Pukhnachev's model on a 
{\em non-rotating} cylinder is controlled by steady states. Our 
results imply that 

\begin{itemize}
\item Eq.~(\ref{eq:Pukh}) with $\omega=0$ 
has for every given mass a unique nonnegative steady state;
\item 
the steady state minimizes energy and attracts all solutions of 
finite entropy;
\item
the distance of solutions from the steady state 
decays no faster than a power law $\sim t^{-\frac{2}{3}}$.  
\end{itemize}
Note that solutions of finite entropy are almost everywhere strictly positive.
The steady state has the shape of a shallow drop hanging from the bottom 
of the cylinder, with a dry region at the top that it meets at zero 
contact angle, see Figure~2 below.  We suspect that the distance from 
the steady state actually behaves as $t^{-\frac{1}{3}}$. This 
conjecture is supported by simulations, and by analogy with aggregation 
processes such as late-stage grain growth in alloys~\cite{LifSly}
and the formation of drops at a faucet, where the growth 
of the grain or droplet is limited by the rate of mass 
transfer through a region of low
density to the region of accumulation.

Our results are motivated by the work of Carrillo and Toscani, who proved 
global convergence to self-similar solutions for the thin film equation
\begin{equation}\label{eq:basic}
u_t + \partial_x [u^n u_{xxx}]=0
\end{equation}
on the real line with $n=1$~\cite{CT}.  In contrast with the recent 
precise convergence results of Giacomelli, Kn\"upfer and Otto~\cite{GKO}, 
our conclusions are qualitative and global. We work in a large class of 
nonnegative strong solutions that contains all positive classical 
solutions, and we do not assume (and do not prove) uniqueness of the 
solutions. Our lower bounds on the distance from the steady state should 
be compared with results of Carlen and Ulusoy~\cite{CU},
who showed for Eq.~(\ref{eq:basic}) with $n=1$ that the 
distance from the self-similar solution satisfies a 
power-law {\em upper bound}. We will give a more detailed
description at the end of this section.  

Thin liquid films have been the subject of rigorous mathematical
analysis since the pioneering article of Bernis and Friedman~\cite{BF}.
A vast body of papers is dedicated to existence of solutions,
regularity, long-time behaviour, finite-time blow-up, and the
interface between wet ($u>0$) and dry $(u=0)$ regions,
see for example~\cite{BG,B2,BertPugh1996,B15,Report,Otto}
and references therein.  An even larger part of the literature
studies the properties of physically relevant solutions through 
asymptotic expansions, numerical analysis, and laboratory experiments.

A fundamental question is {\em well-posedness}: which initial values 
give rise to unique nonnegative solutions that depend continuously on
the data? The difficulty is that solutions of fourth order
parabolic equations generally do not satisfy a maximum principle, 
and linearization leads to semigroups that do not preserve positivity. 
To give  a simple example, the function $u(x,t) = 1 + t\cos(x)$, which 
develops negative values after $t=1$, solves Eq.~(\ref{eq:PDE}) 
(given below) with $n=0$ and $\alpha=1$.
Bernis and Friedman proved that initial-value problems in 
one space dimension have weak solutions in suitable function 
spaces ~\cite{BF}.  A far-reaching technical contribution 
was their use of {\em energy} and {\em entropy} functionals 
that decrease along solutions.  Still, after twenty years, uniqueness 
remains an open problem.

There are many questions surrounding {\em steady states}:
Are they uniquely determined by their mass, and if not, how many are there?
Are they strictly positive, and if not, what is the contact angle 
between wet and dry regions~\cite{LaugPugh}?  
Under what conditions are they stable, do they attract all bounded 
solutions, and what is the rate of convergence?
Since energy decreases along solutions, we expect that steady 
states should correspond to critical points of the energy, that solutions 
should converge to steady states, and that
minimizers of the energy should be asymptotically stable.
However, in the absence of a proper well-posedness
theory, the proof of these statements requires more than a 
standard application of Lyapunov's principle~(as
stated, for example, in~\cite{HS}).

One strategy for proving convergence to equilibrium is to use entropy 
in place of energy.  The basic thin film equation~(\ref{eq:basic})
has a family of entropy functionals of the form 
$ S_\beta(u)= \int_\Omega u^{-\beta}\, dx $ that decrease along 
solutions. For $\beta=n-2$, this was established by Bernis and 
Friedman, for $\beta=n-\frac{3}{2}$  it is due to Kadanoff 
(see~\cite{Kadanoff}), and the range $n-3< \beta < n-\frac{3}{2}$ was
developed independently in~\cite{B2,BertPugh1996}.
Other families of entropies have since been discovered~\cite{JM,Laug}.
In their classical papers, Beretta, Bertsch, Dal Passo and 
Bertozzi, Pugh used these entropies to show that
solutions of Eq.~(\ref{eq:basic}) with $n > 0$ 
on an interval become positive after a finite time and
converge uniformly to their mean~\cite{B2,BertPugh1996}.
Remarkably, this convergence holds for a very broad class of weak 
solutions, about which little else is known.  Several works over 
the last decade have combined energy and entropy methods by deriving 
coupled inequalities for the energy and entropy dissipation. In this 
way, Tudorascu proved exponential convergence to the mean for thin 
films on finite intervals~\cite{Tudorascu}, and 
Carrillo, Toscani~~\cite{CT} and Carlen, Ulusoy~\cite{CU}
have proved global (power-law) convergence to self-similar solutions 
of the thin equation with $n=1$ on the real line~\cite{CT,CU}. 

\begin{figure}[t]
\begin{center}
\includegraphics[height=5.5cm] {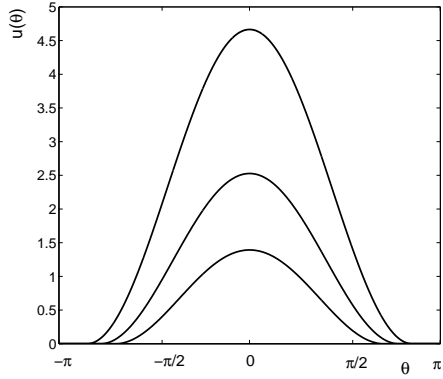}
\end{center}
\caption{\small The energy minimizing steady  state
for Eq.~(\ref{eq:PDE}) with $\alpha=1$ and initial 
data $ u_0 = 0.5, 1, 2$.  The 
minimizer is concentrated if the mass is small, and spreads out
over $(-\pi,\pi)$ as the mass increases.  
}
\end{figure}

A promising approach, first pursued by Otto~\cite{Otto}, 
treats thin-film equations as {\em gradient flows} that 
evolve by steepest descent on the space of nonnegative 
functions of a given mass, endowed with a suitable metric. 
Since Eq.~(\ref{eq:thin-film}) has the form 
$u_t = \left[u^n\left(\frac{\delta E} {\delta u}\right)_x\right]_x$,
a formal computation suggests certain variants of the Wasserstein metric.  
These metrics measure the distance between two mass densities 
as the cost of an optimal transportation plan that pushes one density 
forward to the other.
Gradient flow methods have been very successful 
for the porous medium and fast diffusion equations,
which are second-order degenerate parabolic equations
that share many features of thin films.
There also has been some progress on thin film
equations with $n=1$, where the metric is just the
standard optimal transportation distance~\cite{Otto}.
Thin-film equations with $n>1$ have proved resistant to this 
approach, because the geometry of the relevant Wasserstein 
distance with mobility is not well understood,
and the thin film energy is not geodesically convex~\cite{CLSS}.

Yet another strategy is to linearize about the steady state. 
This linearization is delicate if the steady state 
has a dry region.  Recently, Giacomelli, Kn\"upfer and Otto
have developed a technique specifically
for droplets with zero contact angles~\cite{GKO}.
They prove well-posedness and convergence to the steady state
for initial values in a singularly weighted 
Sobolev space that forces the solution to vanish to a certain 
order at the contact point. So far, these techniques have been
developed for Eq.~(\ref{eq:basic}) with $n=1$.

\bigskip
{\bf Outline of the paper.} \ We study the role of energy-minimizing 
steady states for the
dynamics of 
\begin{equation} \label{eq:PDE}
 u_t + \partial_x\left[u^n\,
(u_{xxx} + \alpha^2u_x -\sin x)\right]
= 0\,,
\quad x \in \Omega = \RR/(2\pi\ZZ)\,
\end{equation}
with $n>0 $ and $\alpha>0$.
Each of these equations describes the evolution of a thin film.
As in Eqs.~(\ref{eq:thin-film}) and~(\ref{eq:basic}), the
exponent $n$ is related to the boundary condition between 
the liquid and the solid,
the first two terms in the parentheses model surface tension, 
and the third term accounts for gravitational drainage.
The $\alpha^2 \partial_x [u^n \partial_x u]$ term is reminiscent of 
the porous-medium equation, but appears with the opposite sign,
resulting in a long-wave instability~\cite{B15}. 
The units of time and length are scaled so that the surface 
tension and gravitational terms appear with coefficient one.
The parameter $\alpha>0$ is a geometric constant,
with $\alpha=1$ for a horizontal cylinder.
Values of $\alpha>1$ appear when inertial 
effects are taken into account~\cite[Eq.~(2.3)]{Kelmanson}. 
In general, the geometric coefficient is a function that
depends on the curvature of the surface~\cite[Eqs. (64) and~(68)]{MCC}.
We are mostly interested in the case where $n=3$ and $\alpha=1$, 
which corresponds to Pukhnachev's model at zero rotation speed,
but find it illuminating to consider also 
other values of $n$ and~$\alpha$.

Many competing definitions of weak solutions
have been proposed. Following Bernis and Friedman, we
define a {\em strong solution} 
of Eq.~(\ref{eq:PDE}) on a finite time interval
$(0,T)$ to be a nonnegative function 
$u\in L^2\bigl((0,T), H^2(\Omega)\bigr)$ that satisfies
\begin{equation} \label{eq:BF-strong}
\int_0^T\int_\Omega
\Bigl\{ u\phi_t - (u_{xx} + \alpha^2\,u 
+ \cos x)\bigl(u^n\phi_x\bigr)_x
\Bigr\}\, dx dt = 0
\end{equation}
for every smooth test function $\phi$ with compact 
support in $(0,T)\times \Omega$. Such solutions 
are believed to be unique.  We will consider strong solutions
that exist for all $t>0$ and satisfy additional bounds on the 
energy and entropy (see Section~\ref{sec:conv}). 
In particular, these solutions are almost everywhere strictly positive.

Our main results concern the convergence of solutions to 
steady states. We start from the {\em energy}
for Eq.~(\ref{eq:PDE}), given by 
\begin{equation} \label{eq:def-E}
\label{EF} E(u) = \frac{1}{2} \int_\Omega
\bigl(u_x^2 - \alpha^2 u^2\bigr)\, dx 
- \int_\Omega u\,\cos x\,dx\, ,
\end{equation}
where the first integral accounts for the surface tension, and 
the second integral for the gravitational potential energy.
Formally, the energy decreases according to 
$\frac{dE(u)}{dt} = - D(u)$, where
\begin{equation}\label{eq:def-D}
D(u) = \int_{\{u>0\}} u^n \bigl(u_{xxx} +\alpha^2u_x
-\sin x\bigr)^2\, dx 
\end{equation}
is the {\em dissipation} associated with Eq.~(\ref{eq:PDE}). 

Note that the energy is not convex for $\alpha>1$,
and not bounded below in $H^1$ for $\alpha\ge 1$. Nevertheless, 
we show in Section~\ref{sec:min} that for every choice of $\alpha>0$ 
the energy has a unique minimizer
among {\em nonnegative} $H^1$-functions of a given mass $M$.
We denote this minimizer by $u^*$. 
The shape of $u^*$ (depending on $\alpha$ and the mass) is described 
precisely in Theorem~\ref{thm:min}.  
The minimization problem is complicated by the fact
that the minimizers can have dry regions,
where the Euler-Lagrange equation is weakened to an inequality. 
We combine a careful analysis of the variational inequality 
with symmetrization techniques.

In Section~\ref{sec:steady}, we show that $u^*$
is the unique steady state of Eq.~(\ref{eq:PDE}) 
with zero dissipation when $\alpha\le 1$. 
For $\alpha>1$, we find also saddle points, as well 
as a continuum of two-droplet steady states that are not 
critical points for the energy in $L^2$. Since the
definition of strong solutions forces steady states to meet any 
dry spot at a zero contact angle, Theorem~\ref{thm:min} implies in 
particular that such steady states exist 
if $M(1-\alpha^2)\le 2\pi$.  For each value of $\alpha$ and $M$, 
there is also a continuum of time-independent solutions
of Eq.~(\ref{eq:PDE}) with compact support
and positive contact angles, which are
analogous to the steady states in~\cite {LaugPugh}. Their role in
the long-term evolution of positive solutions remains open.

Section~\ref{sec:conv} contains the
main convergence result, Theorem~\ref{thm:conv}.
We show that the energy
minimizer $u^*$ is a dynamically stable, locally attractive 
steady state of Eq.~(\ref{eq:PDE}).  For $\alpha\le 1$, 
we show that all solutions that satisfy 
certain energy and entropy inequalities converge to $u^*$.
For $\alpha>1$, the 
convergence holds on a sub-level set of the energy that contains
no other strong steady states.

The entropy methods of~\cite{B2,BertPugh1996,CT,CU,Tudorascu} 
do not apply to Eq.~(\ref{eq:PDE}), because the 
entropy can increase as well as decrease 
along solutions. For steady states with dry regions,
the entropy is not even finite. We combine Lyapunov's method with a 
linear bound on the growth of the entropy to produce a sequence of 
times along which the solution converges weakly to a steady 
state, using a recent argument 
of~\cite{Cheung-Chou}. We pass to convergence 
in norm along the full solution by proving a local coercivity 
estimate on the energy near the minimizer (which becomes 
global for $\alpha\le 1$).

Finally, in Section~\ref{sec:rate},
we turn to the rate of convergence.  
We show that convergence  to the steady state 
is exponential whenever $u^*$ is strictly 
positive. This extends the results of ~\cite{B2,BertPugh1996,Tudorascu}
to examples where the steady state is non-constant
(but note that our proof uses energy in place of entropy).
When $u^*$ has a dry region and $n>\frac{3}{2}$, we show that
solutions cannot approach it faster than
a power law $\sim t^{-\frac{2}{2n-3}}$.  
In the proof, we show that
Kadanoff's entropy $S_{n-\frac{3}{2}}$ grows at most linearly.

All our results are easily adapted to the
long-wave stable case of Eq.~(\ref{eq:PDE})
where the $\alpha^2 \partial_x [u^n \partial_x u]$
term appears with the opposite sign (see~\cite{B15}).
There, the energy-minimizing steady state
is dynamically stable and 
attracts all solutions of finite entropy.
It is strictly positive and exponentially
attractive so long as $M(1+\alpha^2)>2\pi$.
Otherwise, the rate of convergence is limited
by a power-law, at least when $n> \frac{3}{2}$.
For $M(1+\alpha^2)=2\pi$, the minimizer has a
touchdown zero, and for $M(1+\alpha^2)<2\pi$, it has
the shape of a droplet with zero contact angles.

\bigskip
{\bf Acknowledgments.} \ This work
was partially supported by NSERC though a Discovery grant (A.B.,
B.S.) and a postdoctoral fellowship (M.C.). We thank 
Victor Ivrii, Mary Pugh, and Roman Taranets for 
stimulating discussions, and Rick Laugesen 
and Vladislav Vasil'evich  Pukhnachev for 
valuable comments  on an early version of the manuscript.

\section{Energy minimizers}
\label{sec:min}

In this section, we study the energy landscape 
over the space of nonnegative functions of a given mass,
$$
{\cal C}_M=\Bigl\{u\in L^2(\Omega)\mid u\ge 0, \int_\Omega u(x)\, 
dx=M
\Bigr\}\,.
$$
We use two topologies on this space, the usual $L^2$-distance
$||u-v||_2$, and the $H^1$-topology, with distance function
$$
d_{H^1}(u,v) = ||u_x-v_x||_2\,,\quad u,v\in {\cal C}_M\,.
$$ 
Since $u$ and $v$ have the same  mass, their difference
has mean zero, and this distance is equivalent to the usual $H^1$-distance. 
It is clear from Eq.~(\ref{eq:def-E}) that the energy is continuous in $H^1$.
By convention, $E(u)=\infty$ if $u\not\in H^1$.

For $\alpha<1$, $E$ is a positive definite
quadratic form, and hence strictly convex. 
This can be seen either from
the Wirtinger inequality~(see, for example,~\cite[p. 61]{Dac}), or by
writing the energy in terms of the Fourier series of $u$ as
$$
E(u)=\pi\sum_{p\in\ZZ\setminus{0}} (p^2-\alpha^2)|\hat u(p)|^2
-\alpha^2\frac{M^2}{2\pi}- \pi(\hat u(1)+\hat u(-1))\,.
$$
For $\alpha=1$, 
the energy is convex, but not strictly convex on ${\cal C}_M$,
because its Fourier expansion depends linearly 
on $\hat u(\pm 1)$.  For $\alpha>1$, 
convexity is lost. 

We will show that the energy has a unique minimizer on ${\cal C}_M$,
and describe its profile.  Our first lemma shows that a minimizer 
exists.  

\begin{lemma} [Existence of minimizers.]
\label{lem:energy-bound} 
For every $M< \infty$, $E$ attains its minimum on ${\cal C}_M$.
\end{lemma}

\begin{proof} Let $u\in {\cal C}_M$.
Since $u$ is nonnegative and has mean $\frac{M}{2\pi}$, it satisfies
$$
\int_{\Omega} u^2 dx \le M \,\|u \|_{L^{\infty}}\,, \quad
||u||_{L^\infty} \le \frac{M}{2\pi} +
\sqrt{\pi}||u_x||_{L^2}\,.
$$
Inserting these estimates into the functional, we obtain
\begin{eqnarray} 
\nonumber E(u) 
&\ge&  \frac{1}{2}\Bigl( ||u_x||_{L^2} -
\frac{\alpha^2M\sqrt{\pi}}{2}\Bigr)^2 
- \frac{\alpha^4\pi}{8} M^2 -  \Bigl(1+\frac{\alpha^2}{4\pi}\Bigr)M\,,
\label{eq:energy-bound}
\end{eqnarray}
which shows that $E$ is bounded below on ${\cal C}_M$.

Consider a minimizing sequence $\{u_j\}_{j\ge 1}$.
By Eq.~(\ref{eq:energy-bound}), the sequence
is bounded in $H^1$. We invoke the Rellich lemma
and pass to a subsequence
(again denoted by $\{u_j\}$) that converges weakly in
$H^1$ and strongly in $L^2$ to some function $u^*$ in ${\cal C}_M$. Since $E$
is weakly lower semicontinuous on $H^1$, we have 
$$
\inf_{u\in {\cal C}_M}E(u) 
\le E(u^*)
\le \lim_{j\to\infty} E(u_j)\le 
\inf_{u\in {\cal C}_M} E(u) \,,
$$
and conclude that $E$ attains its minimum at $u^*$.
\end{proof}

\bigskip 
The Euler-Lagrange equation for $E$ under the mass constraint 
is given by \begin{equation}\label{eq:EL}
u_{xx} +\alpha^2 u +\cos x = \lambda \,,
\end{equation}
where $\lambda$ is a Lagrange multiplier. 
We need to incorporate also the positivity constraint.
If $u\in {\cal C}_M$, we decompose $\Omega$ according to the value of $u$ 
into the positivity set and the zero set of $u$, defined by
$$
P(u)=\{ x\in\Omega \mid u(x)>0\}\,,
\quad 
Z(u)=\{ x\in\Omega \mid u(x)=0\}\,.
$$

\begin{lemma}[Euler-Lagrange equation
and zero contact angle.] \label{lem:EL}
If $u^*$ minimizes $E$ on ${\cal C}_M$, then 
it solves~(\ref{eq:EL})
on $P(u^*)$. The Lagrange multiplier
is positive and satisfies 
$\lambda \ge \sup \bigl\{ \cos x\mid x\in Z(u^*)\bigr\}$.
Furthermore, $u^*$ is of class ${\cal C}^{1,1}$, and $u^*_x=0$ 
on $\partial P(u^*)$.
\end{lemma}

\begin{proof} Let $\phi$ be a smooth
$2\pi$-periodic test function, and set
$$
u^\eps= \frac{M}{M^\eps} (u^*+\eps\phi)\,,
\quad M^\eps=\int_\Omega \bigl(u^*+\eps\phi\bigr)\, dx\,.
$$
We compute the first variation of $E$ about $u^*$ as 
$$
\frac{d}{d\eps} E(u^\eps) \Big\vert_{\eps=0}
= \int_\Omega 
\bigl( u^*_x\phi_x -\alpha^2 u^*\phi - \phi\cos x
+\lambda \phi \bigr) \, dx\,,
$$
where 
$$
\lambda=-\frac{1}{M}\Bigl(2E(u^*) +\int_\Omega u^*\cos x\, dx\Bigr)\,.
$$
Let $\phi$ be a smooth test function supported in $P(u)$, and set
$$
\eps^0=
\left(\max_{x\in {\rm supp}\phi} \frac{|\phi(x)|}{u^*(x)}\right)^{-1}>0\,.
$$
By construction, $u^\eps(x)\ge 0$ for $|\eps|\le \eps_0$,
and therefore the first variation of $E$ along $u^\eps$ must vanish.
Since this holds for every smooth test function
supported on $P(u^*)$, the Euler-Lagrange equation holds there.

Similarly, the first variation is nonnegative for
every nonnegative test function $\phi$,
because positive values of $\eps$ yield
admissible competitors $u^\eps$ on ${\cal C}_M$. 
This means that $u^*$ satisfies the variational inequality
$$
u_{xx} +\alpha^2 u + \cos x \le \lambda \quad \
\mbox{on}\ \Omega
$$
in the sense of distributions. Taking $\phi\equiv 1$, we that 
$\lambda\ge \alpha^2 \frac{M}{2\pi}>0$,
and by considering nonnegative test functions supported
on $Z(u^*)$, we obtain the claimed inequality for $\lambda$.

To see that the contact angles are zero, assume that $u^*(\tau)=0$.
The variational inequality implies that $u^*(x)- \frac{\lambda+1}{2}(x-\tau)^2$ 
is concave in $x$.  In particular,
the graph of $u^*$ lies below a support line at $x=\tau$. 
Solving for $u^*(x)$, we see that
$$
0 \le u^*(x)\le \frac{\lambda+1}{2}(x-\tau)^2 + b(x-\tau)\,,
$$
where $b$ is the slope of the support line. 
It follows that $b=0$, and $u^*$ is differentiable at $\tau$ with
$u^*_x(\tau)=0$. 
\end{proof}

\bigskip Our next goal is to determine the minimizers of 
$E$ on ${\cal C}_M$ explicitly. At first sight, this appears to be a simple
matter of minimizing over the two free parameters
in the general solution of the Euler-Lagrange equation.
This solution is given by
\begin{equation}\label{eq:EL-sol}
u(x) = \frac{\lambda}{\alpha^2} + u^0(x) + 
A\cos(\alpha x) + B \sin(\alpha x) \,,
\end{equation}
where 
\begin{equation} \label{eq:u-part} 
u^0(x)= 
\left\{ \begin{array}{ll}
-\frac{1}{2} x \sin x\,,\quad & \alpha=1\,,\\[0.1cm]
\frac{1}{1-\alpha^2} \cos x\,, \quad &
\alpha\ne 1\,.
\end{array}\right.
\end{equation}
A moment's consideration shows that the minimizer cannot be 
a strictly positive function given by Eq.~(\ref{eq:EL-sol}),
unless $A=B=0$ and $M>\frac{2\pi}{|1-\alpha^2|}$ (in which case
$\lambda=\frac{M\alpha^2}{2\pi}$).
If $u$ vanishes somewhere in $\Omega$, then 
$\lambda$ depends implicitly on $M$ through a nonlinear equation.
The number of components of the positivity set
is another unknown, the constants $A$ and $B$ may differ from component to
component, and the contact points that form the boundary of the positivity
set contribute additional free parameters. 

We reduce the number of parameters by observing that 
minimizers of $E$ on ${\cal C}_M$ are necessarily
symmetric decreasing on $[-\pi,\pi]$ about $x=0$. To see this, 
let $u^\#$ be the symmetric decreasing rearrangement of 
$u$~\cite[Section 3.3]{LL}. By definition, each
sub-level set $\{x\in\Omega \mid u^\#(x)>s\}$ is an open
interval centered at $x=0$ that has the same measure as the
corresponding set $\{x\in \Omega \mid u(x)>s\}$. Classical
results ensure that $u^\#\in {\cal C}_M$, and that
\begin{equation}\label{eq:rearrange}
||u^\#||_{L^2}=||u||_{L^2}\,,\quad ||u^\#_x||_{L^2}\le
||u_x||_{L^2}\,,\quad \int_\Omega u^\# \cos x \, dx
\ge \int_\Omega u \cos x\, dx \,,
\end{equation}
see~\cite{Kawohl,PSz}. It follows that
$$
E(u^\#)\le E(u)\,.
$$
If $u$ is a minimizer, then $E(u^\#)=E(u)$, and in particular, the last
inequality in (\ref{eq:rearrange}) must hold with
equality. Since the cosine is strictly symmetric  decreasing,
this forces $u$ to be symmetric decreasing as well~\cite[Theorem
3.4]{LL}.

It is now easy to determine the minimizer by solving the
Euler-Lagrange equation  with zero contact angle
boundary conditions on a symmetric interval $(-\tau,\tau)$.
The remaining three parameters are the coefficient $A$, the Lagrange 
multiplier $\lambda$, and the contact point $\tau$.
The next lemma will be used to show that the 
positivity set of a minimizer grows with its mass.

\begin{lemma} [One-sided derivatives.] 
\label{lem:one-sided}
Assume that $u^*$ minimizes the energy on ${\cal C}_M$.
If $P(u^*)=(-\tau,\tau)$ for some $\tau\in (0,\pi]$, then $u^*_{xx}(\tau_-)>0$ 
and $\lambda>\cos\tau$.
\end{lemma}

\begin{proof} Since $u^*$ satisfies
the Euler-Lagrange equation by Lemma~\ref{lem:EL},
it has bounded derivatives of all orders on $(-\tau,\tau)$.
We analyze the sign of the first
non-vanishing one-sided derivative of $u^*$. 
By symmetry, it suffices to consider
the right endpoint at $x=\tau$.
The positivity of $u^*$ implies that 
$u^*_{xx}(\tau_-)\ge 0$. 

Suppose that
$u^*_{xx}(\tau_-)=0$, then $u^*_{xxx}(\tau_-)\le 0$. 
On the other hand, by differentiating Eq.~(\ref{eq:EL}), 
we obtain $u^*_{xxx}(\tau_-)=\sin\tau\ge 0$. 
This leaves only the possibility that $\tau=\pi$.
Differentiating once more, we obtain
$u^*_{xxxx}(\pi_-)=-1$, which is the wrong sign
for $u^*$ to have a minimum at $\pi$. It follows that
$u^*_{xx}(\tau_-)>0$. The claimed inequality for $\lambda$ follows from
Eq.~(\ref{eq:EL}).
\end{proof}

\bigskip 
The following theorem summarizes our results, see Figure~3.

\begin{figure}[t]
\label{fig:mass-tau}
\begin{center}
\includegraphics[height=5.5cm] {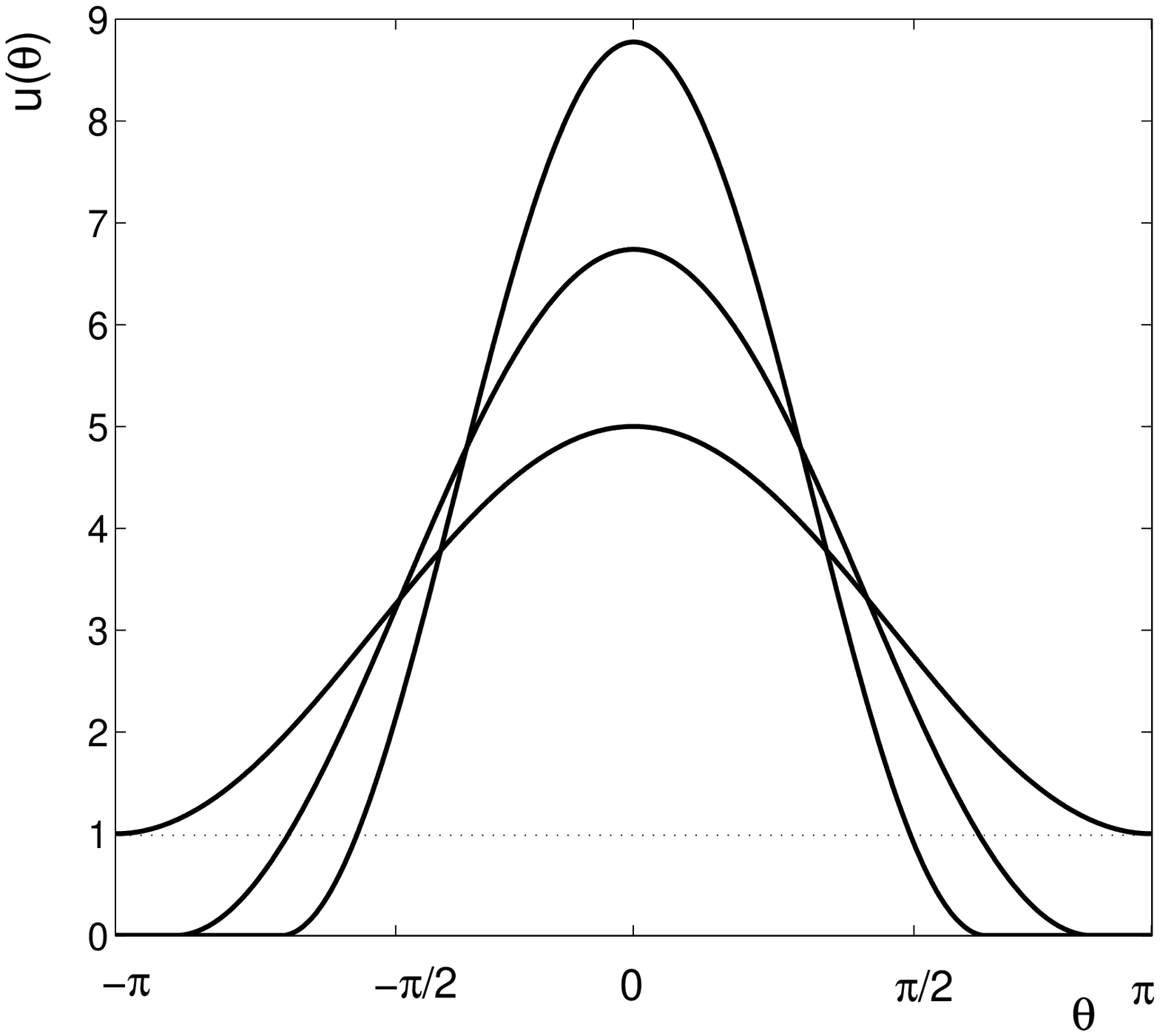}
\quad \quad \quad \quad
\includegraphics[height=5.5cm] {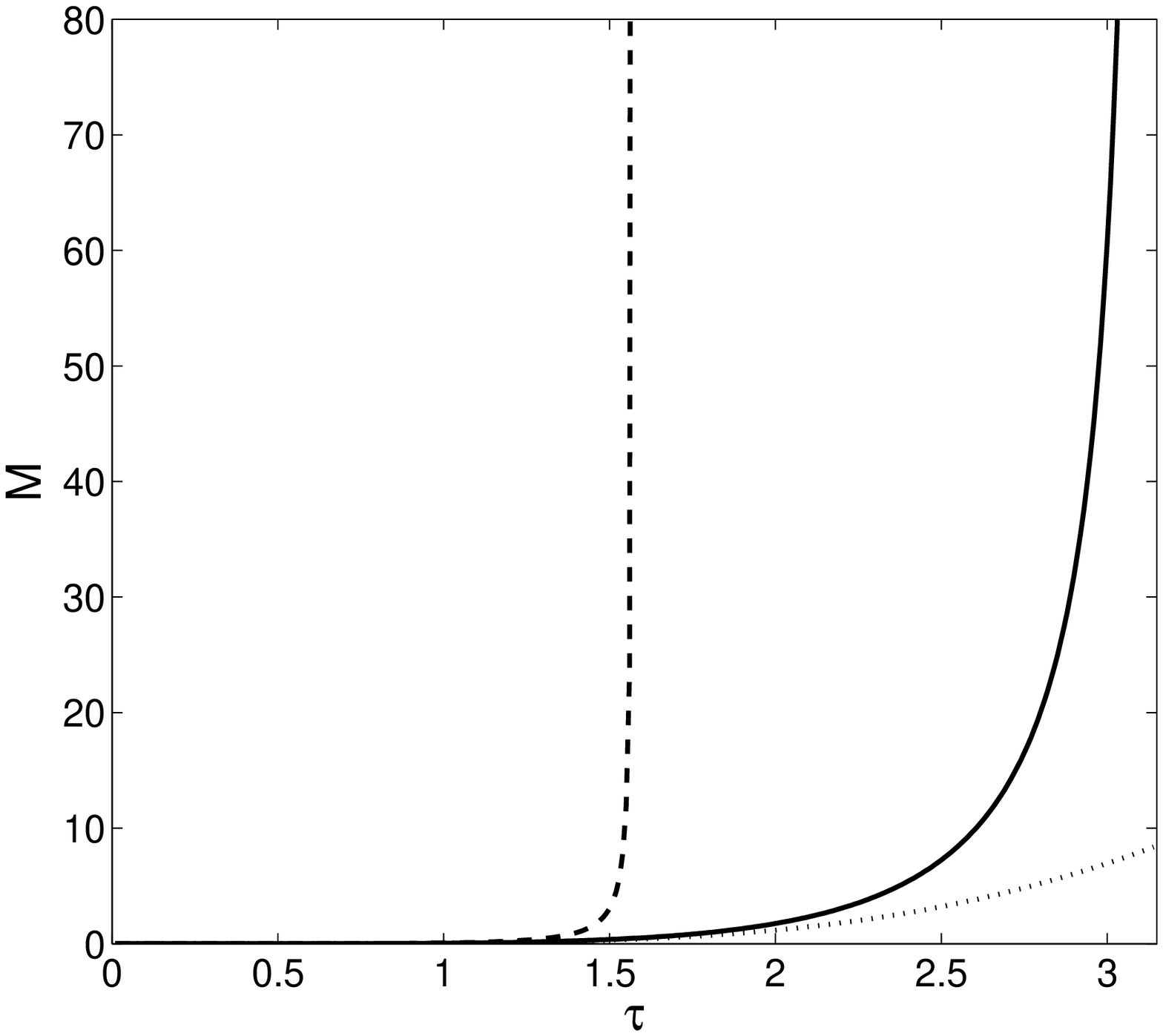}
\end{center}
\caption{\small The energy minimizer for
$\alpha=0.5, 1, 2$ with initial data $u_0=3$
({\em left}).  Mass versus contact point for $\alpha = 2, 1, 0.5$ ({\em right}).}
\end{figure}

\begin{theorem}[Description of the energy minimizers.]
\label{thm:min} Let $E$ be the energy functional in
Eq.~(\ref{eq:def-E}) for some $\alpha > 0$.
For each $M>0$, $E$ has a unique nonnegative
minimizer of mass $M$.  The minimizer is strictly symmetric
decreasing on its positivity set. 
It depends continuously on $M$ in ${\cal C}^{1,1}$,
and increases pointwise with 
$M$ in the sense that for any pair of minimizers $u_1^*,u_2^*$ of mass
$M_1,M_2$,
$$
M_1< M_2\ \Longrightarrow\ u_1^*(x)< u_2^*(x)\,, \quad
x\in P(u_1^*)\,.
$$
If $\alpha<1$ and $M(1-\alpha^2)>2\pi$, then the minimizer is 
strictly positive and given by
\begin{equation}
\label{eq:u-positive}
u^*(x)=\frac{M}{2\pi} + \frac{1}{1-\alpha^2}\cos x\,,
\quad x\in \Omega\,.
\end{equation}
Otherwise, it is given by
\begin{equation}\label{eq:u}
u^*(x) = u^0(x) + A\cos(\alpha x)-
u^0(\tau)- A \cos(\alpha\tau)\,,\quad |x|<\tau
\end{equation}
and vanishes elsewhere. Here, $\tau$ is a  smooth increasing
function of $M$ with
$\tau\cdot \max\{\alpha,1\}<\pi$, 
the function $u^0$ is given by Eq.~(\ref{eq:u-part}),
and $A=\frac{u^0_x (\tau)}{\alpha\sin(\alpha\tau)}$.
\end{theorem}

\begin{proof}  Fix $\alpha> 0$ and $M>0$.  By Lemma~\ref{lem:energy-bound},
there exists a minimizer $u^*$ of mass $M$.
If $u^*$ is strictly positive, then
Eq.~(\ref{eq:EL-sol}) holds for all $x\in\Omega$. Since 
positive minimizers are symmetric about $x=0$,
smooth, periodic, and have mass $M$, we conclude that $\alpha<1$ and 
Eq.~(\ref{eq:u-positive}) holds. In order for $u^*$ to
be nonnegative and symmetric decreasing we must have
$M(1-\alpha^2)\ge 2\pi$.  In that region, $u^*$ is clearly strictly
increasing in $M$.

If, on the other hand, the positivity constraint is active, 
then $u^*$ is symmetric decreasing on some interval $(-\tau,\tau)$ and
vanishes for $|x|\ge \tau$. By Lemma~\ref{lem:EL}, $u^*\in
{\cal C}^{1,1}(\Omega)$ and $u^*_x(\pm\tau)=0$. 
On $(-\tau,\tau)$, $u^*$ is 
given by Eq.~(\ref{eq:EL-sol}). Since $u^*$ and $u^0$ are even,
$B=0$.  The Dirichlet condition at $\tau$ yields 
$$\lambda=-\alpha^2\bigl(A\cos(\alpha\tau)+u^0(\tau)\bigr)\,,
$$
the Neumann condition determines $A$, and we find that
Eq.~(\ref{eq:u}) holds.  We denote this function by $u^*(x;\tau)$.
If $\tau\cdot \max\{\alpha,1\}<\pi$, we claim
that $u^*(x:\tau)$ is indeed nonnegative, symmetric decreasing in $x$,
and strictly increasing with $\tau$. 
To see this, we differentiate
Eq.~(\ref{eq:u}), and use that $u_x^*(\tau;\tau)=0$ 
to obtain $$ \frac{dA}{d\tau}\cdot  \alpha\sin\alpha\tau = -
u^*_{x\tau}(\tau;\tau)= u^*_{xx}(\tau;\tau)>0\,.
$$
By the chain rule, and using once more that $u^*_x(\tau;\tau)=0$,
we have
$$
u^*_\tau (x;\tau) = \frac{dA}{d\tau} \cdot \bigl( \cos(\alpha
x)-\cos(\alpha\tau) \bigr)>0 \quad\mbox{for}\ |x|<\tau\,.
$$
Since $u^*$ vanishes identically when $\tau=0$, this confirms that it
is positive and strictly symmetric decreasing for $|x|<\tau$. 

The mass of $u^*$ is given by $M(\tau)=\int_{-\tau}^\tau u^*(x;\tau)\, dx$.
We use that $u^*_x(\tau;\tau)=0$ to compute
$$
\frac{dM}{d\tau} = \frac{dA}{d\tau}\int_{-\tau}^\tau
\bigl(\cos(\alpha x)-\cos(\alpha\tau)\bigr)\,dx>0\,,
$$
and infer that  we can solve for $\tau=\tau(M)$ as a strictly
increasing smooth function of $M$. By the chain rule and the inverse
function theorem,
$$
\frac{d}{d M} u^*(x;\tau(M)) = \frac{\cos(\alpha x)-
\cos(\alpha\tau)}{ \int_{-\tau}^\tau 
\bigl(\cos(\alpha x')- \cos(\alpha\tau)\bigr)\, dx'}
>0 \,.
$$

It remains to determine the ranges where Eq.~(\ref{eq:u-positive}) and
(\ref{eq:u}) hold. For $\alpha<1$, the energy minimizer 
on ${\cal C}_M$ is unique by the strict convexity of $E$. If 
$M\ge \frac{2\pi }{1-\alpha^2}$, the minimizer
is given by Eq.~(\ref{eq:u-positive}). 
For smaller values of the mass, we use instead Eq.~(\ref{eq:u}),
and compute that $M\to 0$ as $\tau\to 0$ 
and $M\to \frac{2\pi}{1-\alpha^2}$ as $\tau\to \pi_-$. Continuous dependence
on $M$ follows, since Eq.~(\ref{eq:u}) agrees with
Eq~(\ref{eq:u-positive}) at $M=\frac{2\pi}{1-\alpha^2}$.

When $\alpha \ge 1$, the positivity constraint is 
always active, because $E$ is not bounded below without it.
Therefore, the minimizer is given by Eq.~(\ref{eq:u})
on some interval $(-\tau,\tau)$. For $\alpha=1$, we 
must have $\tau<\pi$, because the
particular solution $u^0( x )= -\frac{1}{2}x\sin x$ from
Eq.~(\ref{eq:u-part}) cannot be continued as a differentiable
$2\pi$-periodic function  across $x=\pi$, in violation of
Lemma~\ref{lem:EL}. It is easy to check from Eq.~(\ref{eq:u}) that
$M\to 0$ as $\tau\to 0$ and $M\to\infty$ as $\tau\to \pi$.
For $\alpha>1$, we have that
necessarily $\alpha\tau<1$, since otherwise the function 
defined by Eq.~(\ref{eq:u}) fails to be symmetric
decreasing.  Since $M\to 0$ as $\tau\to 0$ and $M\to\infty$ as $\tau
\to \alpha^{-1}\pi$, the  theorem follows.
\end{proof}

\section{Steady states}
\label{sec:steady}

In this section, we investigate the relationship between
steady states of Eq.~(\ref{eq:PDE})
and critical points of the energy in Eq.~(\ref{eq:def-E}).
For $\alpha\le 1$ and $n\ge 1$, we will show that
the global minimizer $u^*$ determined in Theorem~\ref{thm:min}
is the unique point in ${\cal C}_M$ where the energy
dissipation defined in Eq.~(\ref{eq:def-D}) vanishes,
but for $\alpha>1$ there are additional steady states.  
We start with some definitions.

\begin{definition}{\em Let $u\in {\cal C}_M$ such that 
$E(u)<\infty$, and let $P(u)$ be its positivity set. 
\begin{itemize}
\item 
$u\in H^2(\Omega)$ is a {\em strong steady state} of Eq.~(\ref{eq:PDE}) if
for every smooth $2\pi$-periodic test function $\phi$
$$
\int_\Omega \bigl(u_{xx} + \alpha^2\,u +
\cos x\bigr)\cdot \bigl(u^n\phi_x\bigr)_x \, dx  = 0\,;
$$
\item $u$ is an {\em $L^2$-critical point} 
of the energy on ${\cal C}_M$ if 
every differentiable curve $\gamma: (-\eps_0,\eps_0)\to {\cal C}_M$ 
in $L^2$ with $\gamma(0)=u$ satisfies
$$
E(\gamma(\eps)) \ge E(u) - o\bigl(||\gamma(\eps)-u||_{L^2}\bigr)\,,
\quad \mbox{as}\ \eps \to 0\,.
$$
\end{itemize}
}
\end{definition}

Strong steady states are time-independent
strong solutions in the sense of Eq.~(\ref{eq:BF-strong}).
As elements of $H^2$, they are of class ${\cal C}^{1,1}$ and
meet dry regions with zero contact angles.
The reason why we define critical points in the $L^2$-topology 
rather than in $H^1$ is that the boundary of ${\cal C}_M$ in $H^1$, 
which consist of nonnegative
functions of mass $M$ that vanish at some point on $\Omega$,
contains many  curves that are differentiable in $L^2$ but not in $H^1$.  
By analogy with the subdifferential in convex analysis, 
we ask only for a lower bound on the energy difference 
because $E$ is lower semicontinuous, but not continuous, on ${\cal C}_M$ 
with the $L^2$-norm. 

The following two lemmas relate these notions to the Euler-Lagrange 
equation.

\begin{lemma}[Steady states with zero dissipation.]
\label{lem:steady} Let $u\in {\cal C}_M$,
and define the dissipation $D(u)$ by Eq.~(\ref{eq:def-D}).
If either
\begin{itemize}
\item 
$u$ satisfies Eq.~(\ref{eq:EL}) on each component $C$ of 
$P(u)$ with a constant $\lambda =\lambda(C)$
and with $u=u_x=0$ on $\partial C$, 
\end{itemize}
or, equivalently,
\begin{itemize} \item 
$u\in H^3_{loc} (P(u))\cap H^2(\Omega)$ with $u=u_x=0$ on $\partial P(u)$ 
and $D(u)=0$,
\end{itemize}
then $u$ is a strong steady state.
\end{lemma}

\begin{proof}  We first prove the equivalence of the two
conditions. Assume that
$u\in {\cal C}_M$ satisfies 
Eq.~(\ref{eq:EL}) on each component of $P(u)$.
Since $u_{xxx}=-\alpha^2u_x+\sin x\in H^1(P(u))$,
it follows that $u\in H^3_{loc}(P(u))\cap H^2(\Omega)$, and $D$ vanishes.
Conversely, if $D(u)=0$ then
$u_{xxx} +\alpha^2 u_x-\sin x$ vanishes 
in $L^2_{loc}(P(u))$. This means that 
$u_{xx} +\alpha^2 u+ \cos x$
is locally constant on $P(u)$, i.e., $u$~satisfies Eq.~(\ref{eq:EL})
on each component $C$ of $P(u)$ with a 
constant $\lambda =\lambda(C)$. 

Let $\{C_j\}$ be the collection of connected components of $P(u)$.
If $u$ solves Eq.~(\ref{eq:EL})
on each $C_j$ with some constant $\lambda_j$ 
and $\phi$ is a smooth $2\pi$-periodic test function, then
$$
\Big\vert
\int_{P(u)} (u_{xx} +\alpha^2 u +\cos x)\cdot
(u^n\phi_x)_x\, dx\Big\vert
\le  \sum_{j} \Big\vert 
\int_{C_j} \lambda_j \cdot (u^n\phi_x)_x \, dx\Big\vert 
= 0\,,
$$
showing that $u$ is a strong steady state.
\end{proof}

\bigskip 
Although strong solutions need not be regular enough to justify 
differentiating the energy, for $n\ge 1$
it is not hard to show
(by arguments analogous to~\cite[Lemmas~1 and~2]{GS})
that the dissipation vanishes in all strong 
steady states of Eq.~(\ref{eq:def-D}).
The next lemma generalizes the description of the minimizers
in Lemma~\ref{lem:EL}.

\begin{lemma} [Characterization of critical points.]
\label{lem:critical}
A function $u\in {\cal C}_M$ is an $L^2$-critical point
of the energy if and only if 
there exists a $\lambda\in\RR$ such that
$u$ solves the Euler-Lagrange equation~(\ref{eq:EL}) with this value
of $\lambda$ on every component of $P(u)$ and $u=u_x=0$ on $\partial P(u)$.
\end{lemma}

\begin{proof} Suppose that $u\in {\cal C}_M$
solves Eq.~(\ref{eq:EL}) on $P(u)$,
and that $u=u_x=0$ on $\partial P(u)$. 
For $v\in {\cal C}_M$, we compute the directional derivative
\begin{eqnarray*}
\frac{d}{ds} 
E((1-s)u + sv) \Big\vert_{s=0^+}
&=& \int_\Omega \bigl\{ u_{x}\cdot (v-u)_x  -
(\alpha^2 u +\cos x)\cdot (v-u)\bigr\}
 \, dx\\
&=&- \int_{P(u)} \lambda \cdot (v-u)\, dx - 
\int_{Z(u)} \cos x \cdot (v-u)\, dx\\
&=& \int_{Z(u)} v\cdot(\lambda-\cos x) \,dx\,.
\end{eqnarray*}
Since $E$ agrees with its second order Taylor expansion about
$u$, it can be written as
\begin{equation} \label{eq:Taylor}
E(v)= E(u) + \int_{Z(u)} \!\! v\cdot (\lambda\!-\!\cos x)\, dx
+ \frac{1}{2} \int_{\Omega} \!
\bigl\{(v_x\!-\!u_x)^2-\alpha^2 (v\!-\!u)^2\bigr\}\, dx\,.
\end{equation}
Let $\gamma:(-\eps_0,\eps_0)\to {\cal C}_M$ be a differentiable
curve through $u$ in $L^2$.
Writing $\gamma(\eps)=u + \eps\gamma'(\eps) + o(\eps)$ in 
$L^2$, we see that the nonnegativity of
$\gamma(\eps)$ implies that $\gamma'(\eps)$ vanishes
almost everywhere on the zero set $Z(u)$.
By Eq.~(\ref{eq:Taylor}),
$$
E(\gamma(\eps)) - E(u) \ge \int_{Z(u)} \gamma(\eps)
\cdot (\lambda-\cos x)\, dx
- \frac{1}{2} \int_{\Omega} 
\alpha^2 (\gamma(\eps)-u)^2\, dx = o(\eps)\,,
$$
showing that $u$ is an $L^2$-critical point.

\smallskip
Conversely, assume that $u$ is an $L^2$-critical point.
By considering $\gamma(\eps)=\frac{M}{M^\eps}(u+\eps\phi)$, 
where $\phi$ is a smooth function with support in $P(u)$, we see
that $u$ satisfies the Euler-Lagrange equation on $P(u)$.
Since $u\in H^1$, it is continuous
and vanishes on $\partial P(u)$.  We need to show 
that $u_x$ also vanishes on~$\partial P(u)$.

Consider a connected component $C$
of $P(u)$, and let $\ell$ be its length. 
By Rolle's theorem, $u_x$ vanishes
somewhere on $C$, and by the Euler-Lagrange equation,
$\sup_{C} |u_x| \le \ell \cdot (\lambda+1+||u||_\infty)$.  
In particular, $u$ is Lipschitz continuous on $\Omega$.
We claim that $u$ has one-sided derivatives
at every point $\tau$ with $u(\tau)=0$.
If $\tau$ is the limit of an increasing sequence of zeroes of $u$,
then it follows from the above estimate that $u_x(\tau_-)=0$.
Otherwise, if $\tau$ lies on the right
boundary of a component $C$, its left derivative exists
because $u$ solves the Euler-Lagrange equation on $C$.
Similarly, $u_x(\tau_+)$ exists, and vanishes unless
$\tau$ is the left endpoint of a component of $P(u)$.

Let $\phi$ be a smooth test function
on $\Omega$, and consider variations of the form
$$
u^\eps(x)= \frac{M}{M^\eps} u (x + \eps\phi(x))\,,
\qquad M^\eps = \int_\Omega u(x +\eps\phi(x))\, dx\,.
$$
Since $u\in H^1$, the curve $\gamma:u\mapsto u^\eps$ is differentiable in 
$L^2$ with $\gamma'(0)=\phi u_x$. 
We compute with the chain rule
$$
\frac{d}{d\eps} E(u^\eps)\Big\vert_{\eps=0}
= \int_\Omega \phi_x \cdot\Bigl(
\frac{1}{2} u_x^2 +\frac{1}{2} \alpha^2 u^2 + u\cos x\Bigr)\, dx
- \int_\Omega \phi u \sin x\, dx -\lambda \frac{d}{d\eps} 
M^\eps \Big\vert_{\eps=0}\,,
$$
where 
$$
\lambda= - \frac{1}{M} \Bigl(2 E(u) + \int_\Omega u\cos x\,dx\Bigr)\,,
\quad 
\frac{d}{d\eps} M^\eps\Big\vert_{\eps=0} = \int_\Omega \phi u_x \,.
$$
Setting the first variation equal to zero 
yields the (weak) Beltrami identity  associated with Eq.~(\ref{eq:EL})
(see, for example~\cite[Theorem 2.8]{Dac}).
We next write $P(u)$ as the union of its connected components $C_j$
and integrate the first integral by parts on each component. 
(The number of components may be finite or countable.)
Using that $u$ satisfies Eq.~(\ref{eq:EL}) and vanishes on $\partial P(u)$,
we obtain 
$$
\frac{d}{d\eps} E(u^\eps)\Big\vert_{\eps=0}
= \frac{1}{2} \sum_j \phi  u_x^2 \Big\vert_{\partial C_j}\,.
$$
Since $u$ is an $L^2$-critical point,
this expression vanishes for every smooth function
$\phi$ on $\Omega$.  By concentrating  $\phi$ at one point 
$\tau\in \partial P(u)$, we conclude that 
$|u_x(\tau_-)|=|u_x(\tau_+)|$.

It remains to show that $u_x(\tau_+)=0$.
Suppose for the contrary that
$u_x(\tau_+)=-u_x(\tau_-)= a>0$.
Then $\tau$ is the common boundary point of two components
of $P(u)$. Let $I$ be an open interval that contains
$\tau$ but no other zeroes  of $u$, and consider the variation
$$
u^\eps(x)= 
\left\{ \begin{array}{ll} 
\frac{M}{M^\eps} 
\psi^\eps(u(x))\,,\quad &x\in I\,,\\[0.2cm]
\frac{M}{M^\eps} u(x)\,, & \mbox{otherwise}\,,
\end{array}\right.
\qquad M^\eps= \int_{I^c} u(x)\, dx +
\int_I \psi^\eps(u(x))\, dx\,,
$$
where $\eps < \min_{x \in \partial I} u(x)$, and
$\psi^\eps$ is defined for $y\in [0,\infty)$ by
$$
\psi^\eps(y) =
\left\{ \begin{array}{ll} 
y\,, & y\ge |\eps|\,,\\
\eps\,,\quad            &\eps>  0, \, y\le\eps, \\
\max\{2y +\eps,0\} \,, \quad & \eps < 0, \, y\le -\eps\,.
\end{array}
\right.
$$
Then $\gamma: \eps\mapsto u^\eps$
defines a  curve in ${\cal C}_M$ that is differentiable 
in $L^2$ with $\gamma(0)=u$, and we have 
$$
\gamma'(0)=0\,,\quad 
\frac{d}{d\eps} E(\gamma(\eps))\Big\vert_{\eps=0} 
= \frac{a}{2}\,,
$$
contradicting the assumption that $u$ is an $L^2$-critical point.
\end{proof}

\bigskip Combining Lemma~\ref{lem:steady} with Lemma~\ref{lem:critical},
we see that the profile of a steady state
where the dissipation vanishes is the sum of
profiles of critical points of the energy
whose supports are mutually disjoint.
We next show that all these profiles are symmetric.

\begin{lemma} [Symmetry.]
\label{lem:symmetry} If $u$ is a positive solution
of Eq.~(\ref{eq:EL}) on some interval $C$ with boundary values
$u=u_x=0$, then it is symmetric about $x=0$.
\end{lemma}

\begin{proof} 
Let $u$ be a positive solution
of Eq.~(\ref{eq:EL}) on $C=(\tau_1,\tau_2)$.
Then $u$ is given by Eq.~(\ref{eq:EL-sol}).
For $\alpha =1$, the boundary conditions $u(\tau_j)=u_x(\tau_j)=0$ read
\begin{eqnarray*} 
A \cos(\tau_j) + B \sin(\tau_j) + 
\lambda - \frac{1}{2} x \sin(\tau_j) &=& 0\,,\\
-A \sin(\tau_j) + B \cos(\tau_j) - \frac{1}{2} \sin(\tau_j) - 
\frac{1}{2} x \cos(\tau_j)&=& 0\,.
\end{eqnarray*}
After eliminating $B$,
we see that $\cos\tau_j$ for $j=1,2$ both solve the quadratic equation
\begin{equation} \label{eq:quadratic}
x^2 -2\lambda x -c=0\,,
\end{equation}
where the constant $c$ is determined by  and $A$ and $\lambda$.  
For $\alpha\ne 1$, 
the boundary conditions can be expressed in the single complex equation
$$
e^{-i\alpha\tau_j}(A+iB) = 
-\frac{\lambda}{\alpha^2} - \frac{1}{1 - \alpha^2} \cos(\tau_j)
+ \frac{i}{\alpha(1 - \alpha^2)} \sin(\tau_j)\,. $$
By considering the square modulus $A^2+B^2$,
we see again that $\cos \tau_j$
for $j=1,2$ both solve the quadratic equation
Eq.~(\ref{eq:quadratic}), with a constant that depends
on $A$, $B$, $\alpha$, and $\lambda$.

By Vieta's theorem, the two roots of Eq.~(\ref{eq:quadratic}) add
up to $2\lambda$. Since $\lambda\ge\cos\tau_j$ 
by Lemma~\ref{lem:EL}, we conclude that $\cos \tau_1= \cos\tau_2$.
This leaves two scenarios: The first is that
$C=(-\tau,\tau)$ for some $\tau\in (0,\pi)$, 
and $u$ is given by Eq.~(\ref{eq:u}).
The other scenario is that $C=\Omega\setminus [-\tau,\tau]$
for some  $\tau\in [0,\pi)$, and  $u$ is the
$2\pi$-periodic  continuation of
\begin{equation}\label{eq:u-saddle}
u(x) = u^0(x) + A\cos(\alpha(x+\pi))-
u^0(\tau)- A \cos(\alpha(\tau+\pi))\,,\quad \tau<x<2\pi-\tau\,,
\end{equation}
with coefficient 
$A(\tau)=\frac{u^0_x(\tau)}{\alpha \sin(\alpha(\pi+\tau))}$.
In both cases, $u$ is symmetric about $x=0$.
\end{proof}

\begin{figure}[t]
\begin{center}
\includegraphics[height= 4 cm] {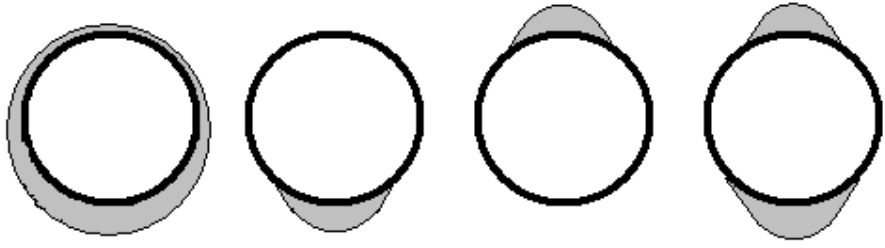}
\end{center}
\label{fig:types}
\caption{\small Types of steady states with zero dissipation 
({\em from left to right}): Smooth film, hanging drop, sitting drop,
and two-droplet steady state.  If $\alpha=1$, the unique
steady state is the energy minimizer,
which has the shape of a hanging drop. If $\alpha<1$, 
the unique energy-minimizing steady state can take 
the shape of a hanging drop (for small values of the mass) or a
smooth film (for larger values of the mass). If $\alpha>1$ 
the energy minimizer is  always a hanging drop. For larger
values of the mass, there are multiple steady states that
include sitting drops and two-droplet steady states.}
\end{figure}

\bigskip 
We conclude from Lemmas~\ref{lem:steady}-\ref{lem:symmetry}
that Eq.~(\ref{eq:PDE}) has four types of strong 
steady states with zero dissipation, see Figure~4.
The first three are $L^2$-critical points 
of the energy, but the last generally is not. 
While critical points are
typically isolated, the fourth type of steady states
can form a one-parameter continuum in ${\cal C}_M$.

\begin{itemize}
\item[I.] {\em Smooth films} are 
positive except possibly for touchdown zeroes. 
They occur as energy minimizers 
if  $\alpha<1$ with $M(1-\alpha^2)\ge 2\pi$,
and as saddle points if $\alpha>1$ with
$M(\alpha^2-1)\ge 2\pi$.
If $\alpha$ is not an integer, they are
symmetric about $x=0$ and given by Eq.~(\ref{eq:u-positive}).
If $\alpha$ is an integer $k>1$ and $M(k^2-1)>2\pi$, 
there are also non-symmetric solutions of the form
$$
u(x)= \frac{M}{2\pi} - \frac{1}{k^2-1} \cos{x}
+ A\cos{k x} + B\sin{kx}\,
$$ 
with $A$ and $B$ small enough.
\item[II.] {\em Hanging drops} are 
given by Eq.~(\ref{eq:u}) on an interval
$(-\tau,\tau)$ with $0<\tau<\pi$,
and have a dry region at the top of the cylinder.
Among them are the energy minimizers
for $\alpha<1$ with $M(1-\alpha^2)<2\pi$, as well as for $\alpha\ge 1$.
\end{itemize}
The other two types occur only for $\alpha>1$:
\begin{itemize}
\item[III.] {\em Sitting drops} are given by Eq.~(\ref{eq:u-saddle})
on an interval $(\tau, 2\pi-\tau)$ with $0<\tau<\pi$, and have a dry region
at the bottom.
They are $L^2$-critical points, but never minimizers of the energy.
\item[IV.]  {\em Two-droplet steady states} are the sum
of a hanging and a sitting drop whose positivity
sets are disjoint.  They do not correspond to $L^2$-critical 
points of the energy unless the value of $\lambda$ 
from Eq.~(\ref{eq:EL}) agrees in the two droplets.
\end{itemize}

The next theorem is illustrated in Figure~5.

\begin{theorem}[Uniqueness of steady states with zero dissipation.]
\label{thm:unique}
Let $M>0$ and $\alpha>0$ be given.
If $\alpha\le 1$, then 
the global minimizer $u^*$ is the unique strong steady 
state of Eq.~(\ref{eq:PDE}) on ${\cal C}_M$ with zero dissipation.
For $\alpha>1$, there exists for each $M>0$
an energy level $E_1> E(u^*)$
such that $u^*$ is the unique strong steady state
with zero dissipation in the sub-level set 
$\{u\in {\cal C}_M\mid E(u)<E_1\}$.
\end{theorem}

\begin{figure}[t]
\begin{center}
\includegraphics[height= 5 cm] {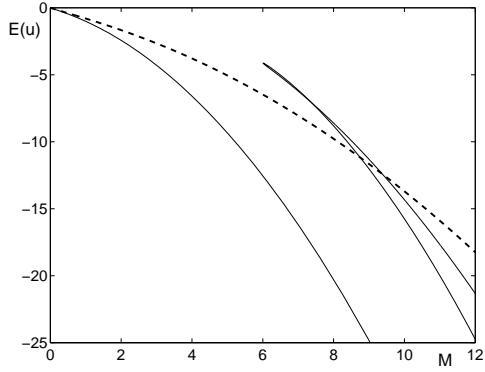}
\end{center}
\caption{\small Energy levels of critical points.
Dashed line: $\alpha=1$. For each given mass, the
energy minimizer is the unique critical point.
Its energy decreases strictly with
the mass. Solid line: $\alpha=\sqrt{2}$. If the mass exceeds a certain
threshold, there is in addition to the global minimizer 
a pair of saddle points. Note that the gap between 
the minimal energy and the energy of the saddles appears to be 
minimal when the saddles first appear, at $M=2\pi$.
Although the energy levels of the two branches cross near $M=8$, 
this is not a bifurcation point.}
\end{figure}

\begin{proof} By Lemmas~\ref{lem:steady}-\ref{lem:symmetry},
every strong steady state where the dissipation vanishes
is the sum of one or two critical points with disjoint 
positivity sets. For $\alpha\le 1$, the convexity of the
energy implies that the unique minimizer $u^*$
(determined in Theorem~\ref{thm:min}) is the unique critical point.
For $\alpha>1$, we use that the minimizer is unique and given 
by a hanging drop, and then use
Lemma~\ref{lem:symmetry} to see that it is isolated
within the set of steady states with zero dissipation.
\end{proof}

\section{Convergence to equilibrium}
\label{sec:conv}

In this section we will prove our main result, that global strong solutions 
of Eq.~(\ref{eq:PDE}) with $\alpha\le 1$ converge strongly in $H^1$
to the unique energy minimizer
of the same mass. We have already described these
steady states in the previous sections.

As described in the introduction, a global strong solution
of Eq.~(\ref{eq:PDE})
is a nonnegative function in $L^2_{loc}((0,\infty),H^2(\Omega)\big)$
that satisfies Eq.~(\ref{eq:BF-strong})
for every smooth test function with compact support 
in $(0,\infty)\times \Omega$.  We consider only strong solutions that 
additionally satisfy a linear bound on the  $H^2$-norm,
\begin{equation} \label{eq:global-strong}
\int_0^T\int_\Omega u_{xx}^2\, dx dt \le 
A + B T\,
\end{equation}
with some constants $A,B$, and the 
{\em energy inequality}
\begin{equation} \label{eq:energy-integrated}
E(u(\cdot, T)) + \int_0^T\int_{P(u)} u^n 
\bigl(u_{xxx}+\alpha^2 u_x 
 - \sin x \bigr)^2 
\, dx dt \le E(u(\cdot,0)) \,.
\end{equation}
To convince the reader that this class of solutions is not empty, we 
paraphrase the existence theory of Bernis and Friedman for $n>2$.
We note in passing that the method has been extended 
to the entire range $n>0$, and that it implies much 
stronger existence and regularity 
results~\cite{B2, BertPugh1996, Report}.  

The basic version of the methods exploits that the entropy 
$$
S(u)=\int_\Omega u^{-n+2}\, dx
$$ 
formally decreases under Eq.~(\ref{eq:basic}) with $n>2$.
Note that $S(u)<\infty$  implies that $u$ can vanish 
only on a set of measure zero.
Strictly speaking, $S$ is not an entropy for Eq.~(\ref{eq:PDE}), 
because  it may increase as well as decrease along solutions,
due to the presence of the long-wave instability
and the gravitational drainage term.
These terms can be accommodated with a technique of
Chugunova, Pugh and Taranets~\cite{Report}, as follows.
Along smooth solutions, 
\begin{equation} \label{eq:naked-entropy}
c_n^{-1} \frac{d S(u)}{dt} 
= - \int_\Omega u_{xx}^2\, dx  +\int_\Omega 
\bigl(\alpha^2u_x^2 + u\cos x\bigr)\, dx\,,
\end{equation}
where $c_n=(n-2)(n-1)$.  
Integrating over time and bounding the second integral
with Eq.~(\ref{lem:energy-bound}), one obtains 
\begin{equation} \label{eq:entropy-integrated}
S(u(\cdot, T))+
c_n \int_{0}^{T}\!\! \int_\Omega u_{xx}^2 \, dx \, dt
\le S(u(\cdot, 0)) + KT \,,
\end{equation}
where $K$ depends only on the energy and the mass.
In particular, classical 
solutions satisfy Eq.~(\ref{eq:global-strong}).
A similar computation shows that the energy
inequality~(\ref{eq:energy-integrated})
is an identity for classical solutions.

Strong solutions that satisfy Eqs.~(\ref{eq:global-strong}) 
and (\ref{eq:energy-integrated}) are obtained by regularizing
$$
 u_t + \partial_x \left[f_{\epsilon}(u)\,
\partial_x(u_{xx} + \alpha^2\,u+ \cos x)
\right] = 0\,,\quad x \in \Omega\,,
$$
for example with $f_\eps (z) = z^n + \eps$~\cite{BF}.
The regularized equations
are known to have unique smooth solutions that exist
for all times $t>0$. For these solutions,
the energy inequality~(\ref{eq:energy-integrated}) holds with equality,
and the entropy inequality~(\ref{eq:naked-entropy}) 
holds for the corresponding entropy functional with the 
the same constants $c_n$ and $K$.
By the usual compactness arguments, there exists
a sequence $u^{\eps_j}$ that converges uniformly on compact
time intervals to a nonnegative strong solution $u$ on 
$(0,\infty)\times\Omega$, which is smooth wherever it is positive.
Moreover, $u^{\eps_j}_{xx}$ converges 
to $u_{xx}$ weakly in $L^2(0,T)\times \Omega$ 
for each $T>0$, and $u^{\eps_j}_{xxx}$ converges 
to $u_{xxx}$ weakly in $L^2_{loc}(P(u))$. 
In the limit, the energy and entropy 
inequalities in Eqs.~(\ref{eq:energy-integrated})
and~(\ref{eq:entropy-integrated}) remain intact, because
the double integrals are weakly lower semicontinuous
due to their convexity in the highest derivative.

\bigskip We turn to the long-time behaviour of 
solutions.  Lyapunov's principle says that a
dissipative dynamical system should converge towards the set 
of critical points of the energy. Thin film equations present
two difficulties: Since well-posedness has not been 
settled, we are not sure how to view them as dynamical systems,
and lower bounds on the dissipation are
not easy to obtain.  

The next lemma bounds the distance from the minimizer 
in terms of the energy.  Since energy decreases along solutions
of Eq.~(\ref{eq:PDE}), the lemma implies that implies that 
$u^*$ is dynamically stable in the sense of Lyapunov.

\begin{lemma}[local coercivity.]
\label{lem:coerce}
Let $\alpha>0$ and $M>0$ be given, and let $u^*$ be the
energy minimizer on ${\cal C}_M$ 
obtained in Theorem~\ref{thm:min}.  Then
$$
\sup \left \{ d_{H^1}(u,u^*) \ \Big\vert \ u\in {\cal C}_M,
E(u)\le E(u^*)+\Delta E\right\} \longrightarrow  0\quad (\Delta E \to 0)\,.
$$
For $\alpha<1$, we have the explicit estimate
\begin{equation}\label{eq:coerce}
d_{H^1}(u, u^*) \le \left(\frac{2\Delta E}{1-\alpha^2} \right)^{1/2}\,
\end{equation}
for all $u$ with $E(u)\le E(u^*)+\Delta E$.
\end{lemma}
\begin{proof} We argue by contradiction.
If the conclusion fails, then
there exists a minimizing sequence
$\{u_j\}$ such that $\inf_j d_{H^1}(u_j,u^*)>0$.
By Lemma~\ref{lem:energy-bound}, a subsequence
converges weakly in $H^1$ and strongly in $L^2$ to
the minimizer $u^*$.
But since the minimizer is unique, and all convergent subsequences
have the same limit, the entire sequence
converges. By the expansion of the energy in Eq.~(\ref{eq:Taylor}),
\begin{eqnarray*}
d_{H^1}(u_j,u^*)
&=& 
\left( 2 \Delta E(u_j) - 2 \int_{Z(u^*)} u_j\cdot (\lambda-\cos x)\, dx
+\alpha^2 \int_\Omega  (u_j-u^*)^2\, dx\right)^{\frac{1}{2}}\\
&\to& 0 \quad (j\to\infty)\,,
\end{eqnarray*}
contradicting the choice of the sequence.
For $\alpha<1$, the bound in Eq.~(\ref{eq:coerce}) follows 
immediately from the observation that the linear
term in Eq.~(\ref{eq:Taylor}) is nonnegative, see Lemma~\ref{lem:EL}.
\end{proof}

\bigskip 
For $\alpha=1$, one can take advantage of
the positivity of the linear term in 
Eq.~(\ref{eq:Taylor}) to obtain an explicit estimate 
of the form $ d_{H^1}(u, u^*) \le c_1(\Delta E)^{1/2} +c_2\Delta E$,
where the constants $c_1$ and $c_2$ depend on the mass.
We next construct a sequence of times 
along which the dissipation goes to zero.

\begin{lemma}[Construction of a weakly convergent sequence.]
\label{time-sequence} Let $u$ be a global strong solution of 
Eq.~(\ref{eq:PDE}) that satisfies 
inequalities~(\ref{eq:global-strong}) and~(\ref{eq:energy-integrated}),
and let $E_0$ be its energy at time $t=0$.
There exists a sequence of times $t_j\to\infty$
such that 
$$
\sup_j ||u_{xx}(\cdot, t_j)||_2 <\infty\,,\quad
\lim_{j \rightarrow \infty} D(u(\cdot, t_j))
= 0\,.
$$
\end{lemma}

\begin{proof} Eq.~(\ref{eq:global-strong}) implies that the set
$$
C_1 = \Bigl\{ t\in (0,T)\Big\vert ||u_{xx} (\cdot, t)||_2^2 
\ge 4B\Bigr\}
$$
has measure bounded by 
$\mu(C_1)\le \frac{A}{4B} +\frac{T}{4}$.
Similarly,
for every $\eps>0$, Eq.~(\ref{eq:energy-integrated}) implies that
$$
C_2 = \Bigl\{ t\in (0,T)\Big\vert D(u(\cdot, t)) \ge \eps \Bigr\}
$$
has measure  bounded by $\mu(C_2)\le \frac{E_0-E(u^*)}{\eps}$.
It follows that for $T>\frac{A}{B}+\frac{4(E_0-E(u^*))}{\eps}$,
we can find $t\in [\frac{T}{2},T]$ that lies
neither in $C_1$ nor in $C_2$. The sequence 
$t_j$ is constructed by taking sequences $\eps_j\to 0$
and $T_j\to\infty$.
\end{proof}

\bigskip We combine this lemma with the stability result
from Lemma~\ref{lem:coerce}  to show that $u^*$ is 
in fact asymptotically stable.

\begin{theorem}[Asymptotic stability.]
\label{thm:conv} Let $u$ be a global strong solution of
Eq.~(\ref{eq:PDE}) of mass $M$ and initial energy
$E_0$ constructed by the method of Bernis and Friedman,
and let $u^*$ be the energy minimizer
on ${\cal C}_M$.  If $\alpha>1$, assume in addition that
the sub=level set $\{E\le E_0\}$ in ${\cal C}_M$ contains
no other steady states with zero dissipation.
Then
$$
\lim_{t\to\infty} d_{H^1}(u(\cdot, t), u^*)=0\,.
$$
\end{theorem}

\begin{proof} Let $\{t_j\}$ be the sequence of times constructed
in Lemma \ref{time-sequence}.
Since $\{u(\cdot, t_j)\}$ is uniformly bounded in $H^2$,
there is a subsequence (again denoted $t_j$) that
converges weakly in $H^2$ and strongly in 
$H^1$ to some limit $v\in {\cal C}_M$.  We want to show that $v = u^*$. 

For $\delta>0$, consider the set 
$P_\delta(v)=\{x\in\Omega \mid v(x)>\delta\}$.
Since $u(\cdot, t_j)$ converges uniformly to $v$, we have
that $u(x,t_j)>\frac{\delta}{2}$ on 
$P_\delta(v)$ for $j$ sufficiently large, and it follows that
$$
\int_{P_\delta(v)} 
\bigl(u_{xxx}(\cdot, t_j) +\alpha^2 u_x(\cdot, t_j)-
\sin x \bigr)^2\, dx \le \frac{2}{\delta} D(u(\cdot, t_j))\to 0\,.
$$
Since we already know that $u_x(\cdot, t_j)$ converges to 
$v_x$ strongly in $L^2$, this means that
$u_{xxx}(\cdot, t_j)$ converges
to $\sin x -\alpha^2 v_x$ strongly in $L^2(P_\delta(v))$. 
The limit agrees with $v_{xxx}$, and we see that
$$
v_{xxx} +\alpha^2v_x-\sin x = 0
\quad \mbox{on} \ P_\delta(v)\,.
$$
Since $\delta$ was arbitrary and $v\in H^2$ by construction,
it follows from Lemma~\ref{lem:steady} that $v$
is a strong steady state of Eq.~(\ref{eq:PDE}).
Since $E(v)\le E(u(\cdot, 0))$ and $D(v)=0$, we conclude with
Theorem~\ref{thm:unique} that $v=u^*$.

We next observe that $E(u(\cdot, t_j))\to  E(u^*)$ by the
continuity of the energy in $H^1$. Since the energy decreases
monotonically along solutions by 
Eq.~(\ref{eq:energy-integrated}), we have
$$
\lim_{t\to\infty} E(u(\cdot, t))= E(u^*)\,,
$$
and the claimed convergence follows with Lemma~\ref{lem:coerce}.
\end{proof}

\section{Rate of convergence}
\label{sec:rate}

In the final section, we consider the rate of convergence
to steady states. We will show that strictly
positive energy minimizers are exponentially attractive, while
steady states that have zeroes can be approached at most at a polynomial rate.
The reason is that the entropy inequality in
Eq.~(\ref{eq:entropy-integrated}) limits the rate at which
the solution can converge to zero on a subset of $\Omega$.
To obtain the strongest lower bound, we will use
Kadanoff's entropy
\begin{equation}
\label{eq:S-Kadanoff}
S(u) = \int_{\Omega} u^{-n+\frac{3}{2}}\, dx\,.
\end{equation}
One can verify by direct calculation (involving repeated integration 
by parts~\cite{B2,JM}) that classical positive solutions of the 
thin-film equation Eq.~(\ref{eq:basic}) with $n\ne \frac{3}{2}$ satisfy
\begin{equation} \label{eq:Kadanoff}
c_n^{-1} \frac{dS(u)}{dt} 
= \int_\Omega u^{-\frac{1}{2}} u_x u_{xxx}\, dx
= -4 \int_\Omega u^{\frac{1}{2}}
\Bigl( (u^{\frac{1}{2}})_{xx}\Bigr)^2\, dx < 0\,,
\end{equation}
where $c_n = \bigl(n-\frac{3}{2}\bigr)\bigl(n-\frac{1}{2}\bigr)$.
This is a special case of Eq. (2.13) in~\cite{B2}. 
For Eq.~(\ref{eq:PDE}), Kadanoff's
entropy can grow at most linearly with time:

\begin{lemma} [Entropy bound.]
\label{lem:entropy} Fix $n>\frac{3}{2}$,
and let $S$ be given by Eq.~(\ref{eq:S-Kadanoff}).
Let $u_0\in {\cal C}_M$  be an initial value
of finite energy and entropy.
Then the global strong solution of Eq.~(\ref{eq:PDE})
constructed by the method of Bernis and Friedman satisfies
\begin{equation}\label{eq:entropy}
S(u(\cdot, t)) \le S(u(\cdot,0)) + K_0t\,,
\end{equation}
where $K_0$ depends on the mass and the initial energy.
\end{lemma}

\begin{proof} If $u$ is a positive classical solution
of Eq.~(\ref{eq:PDE}), we differentiate the entropy
and integrate by parts to obtain
$$
c_n^{-1} \frac{dS(u)}{dt} = \int_\Omega u^{-\frac{1}{2}} u_x
u_{xxx}\, dx + \alpha^2 \int_\Omega
u^{-\frac{1}{2}} u_x^2\,dx - \int_\Omega
u^{-\frac{1}{2}}u_x\sin x\, dx\,,
$$
where $c_n= \bigl(n\!-\!\frac{3}{2}\bigr)\bigl(n\!-\!\frac{1}{2}\bigr)$. 
The first summand we rewrite with the help of  Eq.~(\ref{eq:Kadanoff}).
The second summand we integrate by parts
$$
\int_\Omega  u^{-\frac{1}{2}} u_x^2\, dx
= 2 \int_\Omega u (u^{\frac{1}{2}})_{xx}\, dx\,,
$$
and combine it with the first by completing the square.
This produces a remainder term of the 
form $\frac{\alpha^4}{4}\int_\Omega u^{\frac{3}{2}}\, dx$.
The third summand we integrate by parts as well.  We arrive at
$$
c_n^{-1}\frac{dS(u)}{dt} = - \int_\Omega
u^{\frac{1}{2}} \left( 2( u^{\frac{1}{2}})_{xx}-
\frac{\alpha^2}{2} u^{\frac{1}{2}}\right)^2\,dx +
\frac{\alpha^4}{4} \int_\Omega u^{\frac{3}{2}}\,dx
+2\int_\Omega u^{\frac{1}{2}}\cos x\, dx\,.
$$
By Lemma~\ref{lem:energy-bound}, the last two integrals
are bounded by a constant that depends only 
on the mass and the energy. Integrating along the solution,
we see that Eq.~(\ref{eq:entropy}) holds for classical solutions. 

By the same computation, Eq.~(\ref{eq:entropy}) holds for
the solutions of a suitably regularized equation 
with the correspondingly regularized entropy 
and with the same constant $K_0$.  
Since the strong solution is a uniform limit 
of such solutions, the entropy converges as well, and
the claim follows.
\end{proof}

\begin{theorem}[Bounds on the rate of convergence.]
\label{thm:decay} Consider Eq.~(\ref{eq:PDE}) with 
parameters $n>0$ and $\alpha>0$, and
set $\beta=n-\frac{3}{2}$. Let $u$ be a solution of mass $M$ that
satisfies the energy and entropy inequalities 
in~(\ref{eq:energy-integrated}) and (\ref{eq:entropy}).
Assume that $u$ converges  in $H^1$ to the 
energy-minimizing steady state $u^*$ of mass $M$.

\begin{itemize}
\item If $n>\frac{3}{2}$ and $u^*$ vanishes on a set of positive length $L$,
then
$$
d_{H^1}(u(\cdot, t), u^*) \ge 
\frac{1}{\sqrt{\pi}}\cdot \Bigl(\frac{L}{S_0+K_0t}\Bigr)^{\frac{1}{\beta}}\,;
$$
\item
if $n>2$ and  $u^*$ vanishes quadratically at a point, then there
exist positive constants $K_1$ and $K_2$ (depending on the
initial energy and entropy) such that
$$
d_{H^1}( u(\cdot,t),u^*)\ge
\bigl(K_1+K_2t\bigr)^{-\frac{2}{2\beta-1}}\,;
$$
\item if $\alpha<1$ and $u^*$ is strictly positive, then
$$
d_{H^1}(u(\cdot, t),u^*)\le K_3 e^{-\mu t}\,
$$
for some  constant $K_3$ (depending on $u$), 
where $\mu = (1-\alpha^2)(\min u^*)^n$.

\end{itemize}
\end{theorem}

\begin{proof} 
If $Z(u^*)$ has measure $L>0$,  we estimate
$$
S(u(\cdot, t)) \ge L
\Bigl(\ \sup_{x\in Z(u^*)} u(x,t)\Bigr)^{-\beta}
\ge L\cdot ||u(\cdot, t)-u^*||_{L^\infty}^{-\beta}\,.
$$
Since $u$ and $u^*$ have the same mass,
we have $||u(\cdot, t)-u^*||_{L^\infty}\le  \sqrt{\pi}\,
d_{H^1}(u(\cdot,t),u^*)$, 
and the first claim follows from the bound on the
entropy in Eq.~(\ref{eq:entropy}). 

If $u^*$ has a zero of order $\gamma>\frac{1}{\beta}$, we consider
the interval $I$ of length $L$ centered at that point and obtain with
the same calculation as for the first case that
$$
||u(\cdot, t)-u^*||_{L^\infty} 
\ge   \sup_{x\in I} u(x, t)- \sup_{x\in I} u^*(x) 
\ge  \Bigl(\frac{L}{S_0 + K_0t}\Bigr)^\frac{1}{\beta}
- O(L^{\gamma})
$$
as $L\to 0$.
Choosing $L=\eps\cdot(S_0+K_0t)^{-\frac{1}{\beta\gamma-1}}$, we see that
for $\eps>0$ sufficiently small
$$
||u(\cdot, t)-u^*||_{L^\infty} \ge
(K_1+K_2t)^{-\frac{\gamma}{\beta\gamma-1}}
$$
with some constants $K_1, K_2>0$. Setting $\gamma=2$ 
and adjusting the constants we obtain the second claim.

If $u^*$ is strictly positive, then $\alpha<1$ by Theorem~\ref{thm:min}. 
By Eq.~(\ref{eq:Taylor}),
\begin{eqnarray*}
E(u(\cdot,t))-E(u^*)&=&\frac{1}{2} \int_\Omega 
\bigl( (u_x-u^*_x)^2-\alpha^2 (u-u^*)\bigr)^2 \, dx\\
& =&
\pi \sum_{p\in\ZZ\setminus{0}} (p^2-\alpha^2)|\hat u(p)-\hat{u^*}(p)|^2\,.
\end{eqnarray*}
Since $u$ converges to $u^*$ in  $H^1$, 
there exists a time $t_0$ such 
that $\min u(\cdot, t)>0$ for all $t> t_0$.
At all later times, $u$ is a
strictly positive, classical solution that can be differentiated as
often as necessary.  Since $u^*$ solves the Euler-Lagrange 
equation~(\ref{eq:EL}), the dissipation satisfies
\begin{eqnarray*}
\frac{d}{dt} E(u(\cdot, t) ) &=&  -\int_\Omega u^n 
\bigl((u-u^*)_{xxx}+\alpha^2 (u-u^*)_x \bigr)^2\, dx\\
&\le& - (\min u)^n \cdot \int_\Omega \left [
\partial_x((u-u^*)_{xx}+\alpha^2 (u-u^*))\right]^2 
\,dx\\
&=& 
- (\min u)^n \cdot \pi \sum_{p\in\ZZ\setminus\{0\}} 
p^2(p^2-\alpha^2)^2 |\hat u(p)-\hat{u^*}(p)|^2\\
&\le& - \Bigl(\frac{\min u \ }{\min u^*}\Bigr)^n 
\cdot 2\mu \, \bigl( E(u(\cdot, t))-E(u^*)\bigr)\,.
\end{eqnarray*}
In the last two steps, we have used Parseval's identity to rewrite the
integral  in terms of the Fourier coefficients of $u-u^*$, and
estimated the Fourier multipliers by
$$
p^2(p^2-\alpha^2)^2\ge (1-\alpha^2)(p^2-\alpha^2)\,,\quad (p\ne
0)\,.
$$
Since $u(\cdot, t)$ converges uniformly to $u^*$  as $t\to\infty$
by Theorem~\ref{thm:conv},
it follows from Gronwall's lemma that 
$E(u(\cdot, t))-E(u^*)\le Ke^{-2\mu t}$ for some constant $K$.
By Eq.~(\ref{eq:coerce}) of Lemma~\ref{lem:coerce},
this implies the claimed exponential convergence of 
$d_{H^1}(u(\cdot, t),u^*)$.
\end{proof}

\bigskip If $n=\frac{3}{2}$ and $u^*$ vanishes on a set of positive
length, we obtain as in the proof of the
first case of the theorem yields
an exponential bound of the
form $d_{H^1}(u(\cdot, t),u^*)\ge  K e^{-\mu t}$,
where $K$ and $\mu$ depend on the mass, energy, and entropy of the solution.

\begin{figure}[t]
\label{fig:converge}
\begin{center}
\includegraphics[height=5.5cm] {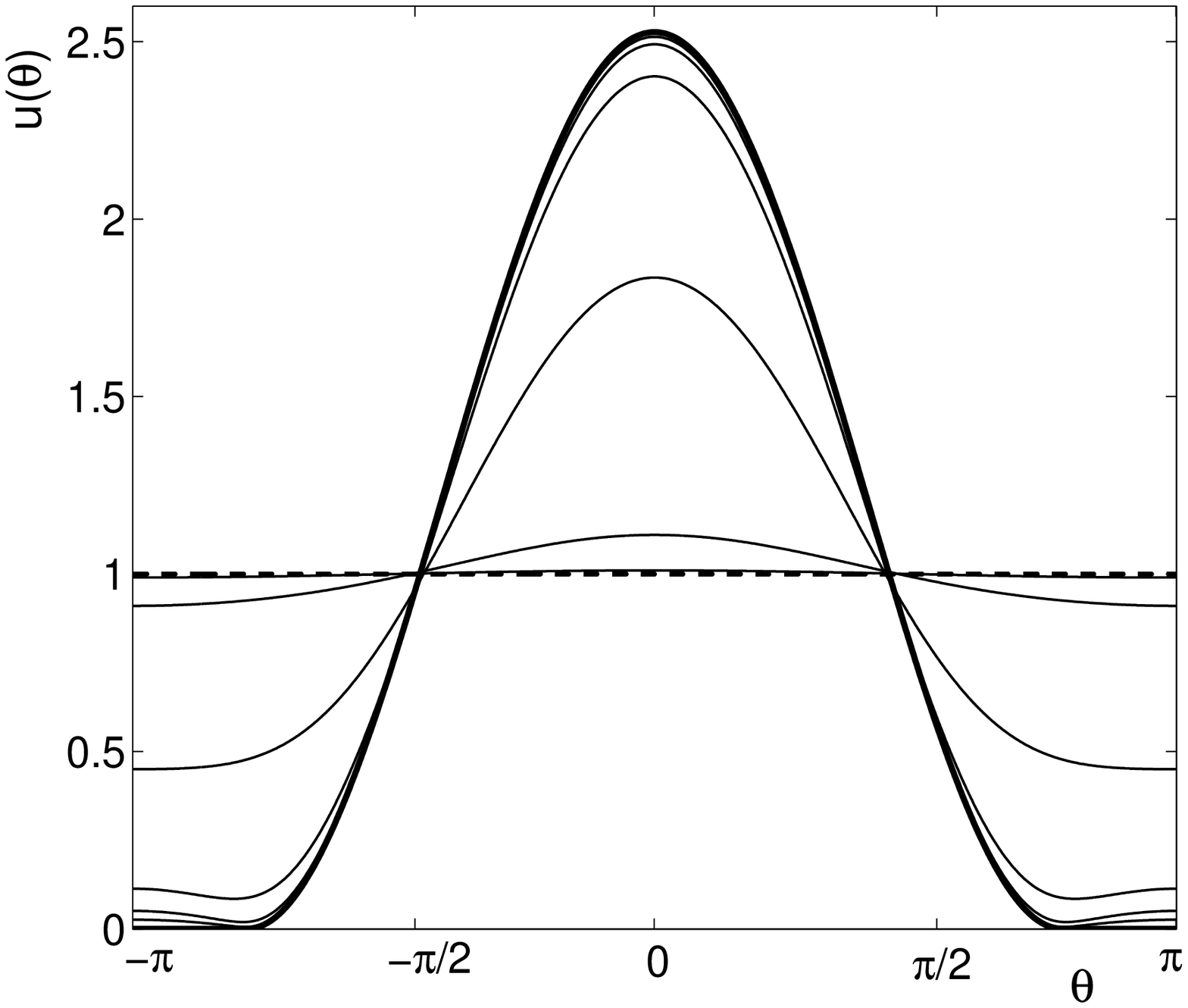}
\quad\quad \quad\quad
\includegraphics[height=5.5cm] {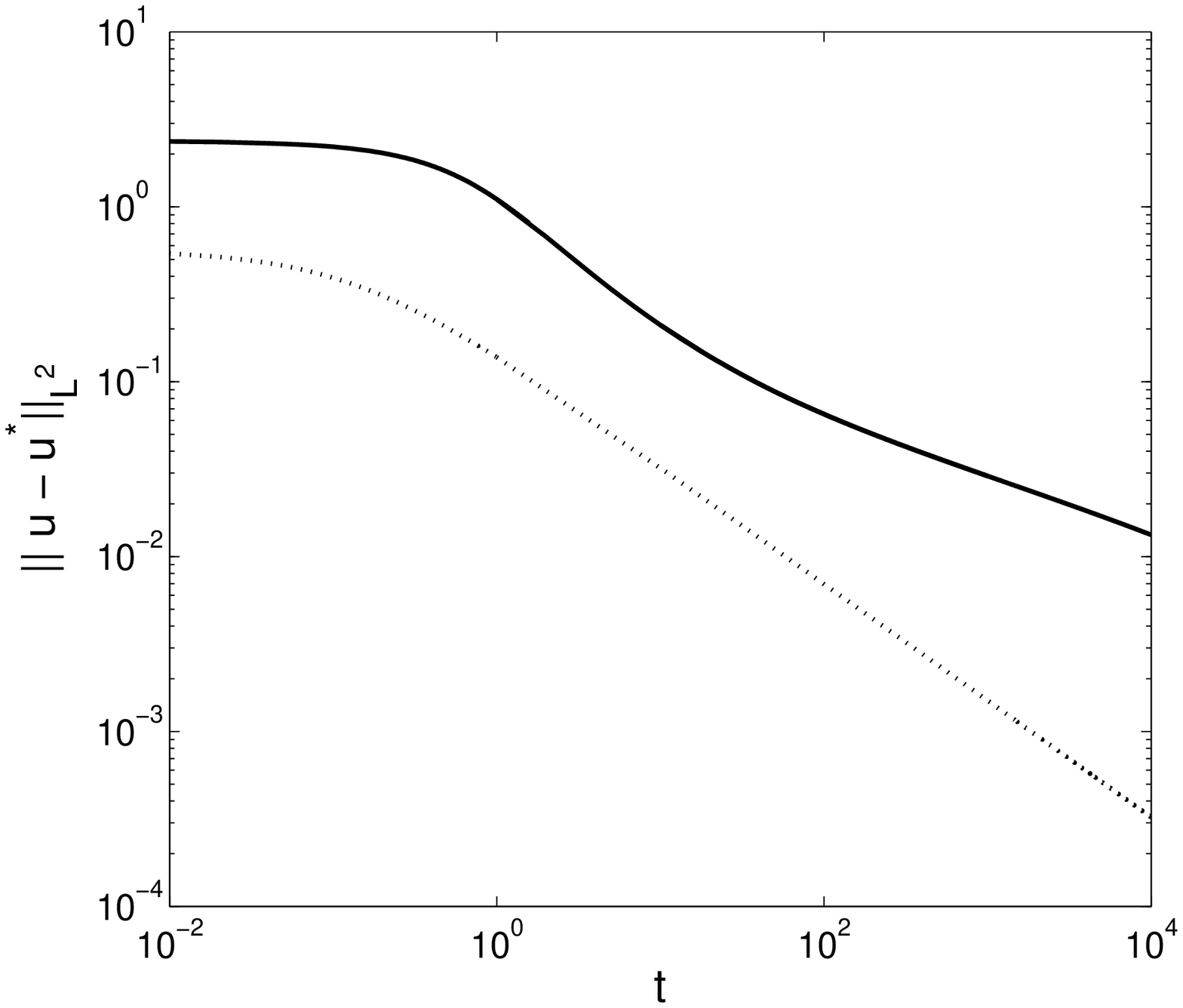}
\end{center}
\caption{\small Evolution of a solution with $\alpha=1$, $n=3$ and
initial data $u_0=1$. Time shots of the numerical
solution at $t = 0, 10^{-2}, 10^{-1}, 1, 10, 10^2, 10^3$
({\em left}). 
$L^2$-distance of the solution from the energy minimizer. The
dashed line shows the lower bound from Eq.~(\ref{eq:Pukh-rate}) ({\em right}).}
\end{figure}

\bigskip 
{\bf Summary (Pukhnachev's model on a stationary cylinder.)}\ 
{\em 
Let $u$ be a global strong solution of
$$
 u_t + \partial_x\left[u^3\,
\partial_x (u_{xx} + u+\cos x)\right]
= 0\,, \quad x \in \RR/(2\pi\ZZ)\,
$$
constructed by the method of Bernis and Friedman,
and let $u^*$ be the unique nonnegative energy minimizer of the
same mass $M$.
Then
$$
\lim_{t\to\infty} d_{H^1}(u(\cdot,t),u^*)=0\,.
$$
The minimizer is a droplet with zero contact angles and profile
$$
u^*(x) = -\frac{1}{2}\bigl(x\sin x - \tau\sin\tau)
+ \frac{1}{2} (1-\tau\cot\tau)(\cos\tau-\cos x)\,,
\quad |x|<\tau\,.
$$
The contact point $\tau$ is a continuous, strictly increasing
function of the mass, with $\tau\to 0$ as $M\to 0$
and $\tau\to\pi$ as $M\to\infty$. If, additionally,
$$
\int_\Omega \bigl(u(x, t)\bigr)^{-\frac{3}{2}}\,dx \le S_0 + K_0 t \,,
$$
then
\begin{equation}
\label{eq:Pukh-rate}
d_{H^1}(u(\cdot, t), u^*) \ge \frac{1}{\sqrt{\pi}}\cdot\Bigl(
\frac{2(\pi\!-\!\tau)}{S_0+K_0t}\Bigr)^{\frac{2}{3}}\,.
\end{equation}
}


\begin{thebibliography}{123}

\addcontentsline{toc}{section}{References}

\bibitem{BG}
{J. Becker and G. Gr\"un.}
\newblock The thin-film equation: recent advances and some
new perspectives.
\newblock {\em J. Phys. Condens. Matter}, {\bf 17}:S291-S307, 2005.

\bibitem{Benilov3}
{E.~S. Benilov, M.~S. Benilov, and N.~Kopteva.}
\newblock Steady rimming flows with surface tension.
\newblock {\em J. Fluid Mech.}, {\bf 597}:91--118, 2008.

\bibitem{Benilov1}
{E.~S. Benilov, S.~B.~G. O'Brien, and I.~A. Sazonov.}
\newblock A new type of instability: explosive disturbances in a liquid film
  inside a rotating horizontal cylinder.
\newblock {\em J. Fluid Mech.}, {\bf 497}:201--224, 2003.

\bibitem{B2}
{E.~ Beretta, M.~Bertsch, and R.~Dal~Passo.}
\newblock Nonnegative solutions of a fourth-order nonlinear degenerate
  parabolic equation.
\newblock {\em Arch. Rational Mech. Anal.}, {\bf 129(2)}:175--200, 1995.

\bibitem{BF}
{F.~Bernis and A.~Friedman.}
\newblock Higher order nonlinear degenerate parabolic equations.
\newblock {\em J. Differential Equations}, {\bf 83(1)}:179--206, 1990.

\bibitem{Kadanoff}  
{A. L. Bertozzi, M. Brenner, T. F. Dupont, and L. P Kadanoff.}
\newblock Singularities and similarities in interface flows. 
\newblock Trends and perspectives in applied mathematics, 155--208, 
Appl. Math. Sci., 100, Springer, New York, 1994. 

\bibitem{BertPugh1996} {A.~L. Bertozzi and M.~C.~Pugh.}
\newblock The lubrication approximation for thin viscous films: regularity and
long-time behaviour of weak solutions.
\newblock {\em Comm. Pure Appl. Math.}, {\bf 49(2)}:85--123, 1996.

\bibitem{B15}
{A.~L. Bertozzi and M.~C. Pugh.}
\newblock Long-wave instabilities and saturation in thin film equations.
\newblock {\em Comm. Pure Appl. Math.}, {\bf 51(6)}:625--661, 1998.

\bibitem{CLSS}
{J. A. Carrillo, S. Lisini, G. Savar\'e, and D. Slep\v{c}ev.}
\newblock Nonlinear mobility
continuity equations and generalized displacement convexity.
\newblock {\em J. Functional Anal.} {\bf  258}:1273-1309 (2010).

\bibitem{CT} J. A. Carrillo and G. Toscani.
\newblock  Long-time asymptotics for strong solutions of 
the thin film equation. 
\newblock{\em Comm. Math. Phys.} {\bf  225}:551-571 (2002).

\bibitem{CU} {E. A. Carlen and S. Ulusoy.}
\newblock Asymptotic equipartition and long time behavior of
solutions of a thin-film equation.
\newblock {\em J. Differential Equations} {\bf 241 (2)}:279--292, 2007.

\bibitem{Cheung-Chou}{K.-L. Cheung and K.-S. Chou.}
\newblock{On the stability of single and multiple droplets 
for equations of thin film type}.
\newblock{\em Nonlinearity} {\bf  23}, 3003–3028 (2010).

\bibitem{ChKP09}
{M.~Chugunova, I.~M. Karabash, and S.~G. Pyatkov.}
\newblock On the nature of ill-posedness of the forward-backward heat equation.
\newblock {\em Integr. Eqn. Oper. Theory}, {\bf 65}:319--344 (2009).


\bibitem{Report}
{M.~Chugunova, M.~C. Pugh, and R.~M. Taranets.}
\newblock Nonnegative Solutions
for a Long-Wave Unstable Thin Film Equation with Convection.
\newblock {\em SIAM J. Math.  Anal.} {\bf 42}:1826-1853 (2010).

\bibitem{Dac} B. Dacorogna.
\newblock Introduction to the Calculus of Variations.
\newblock Imperial College Press, second edition 2009.

\bibitem{GKO}
L. Giacomelli, H. Kn\"upfer, and F. Otto.
\newblock Smooth zero-contact-angle solutions to a thin-film equation 
around the steady state.
\newblock {\em  J. Differential Equations} {\bf 245 (6)}: 1454--1506, 2008.

\bibitem{GS}
Dan Ginsberg and Gideon Simpson,
\newblock 
Analytical and numerical results on the positivity 
of steady state solutions of a thin film equation.
\newblock Preprint {\tt arXiv:1101.3261}, 2011.

\bibitem{HS} M. W. Hirsch and S. Smale.
\newblock Differential equations, dynamical systems, and linear algebra.
\newblock Academic\ Press, New York-London, 1974.

\bibitem{John} {R.~E. Johnson.}
\newblock Steady state coating flows inside a rotating horizontal cylinder.
\newblock {\em J. Fluid Mech.}, {\bf 190}:321--322, 1988.

\bibitem{JM} {A.~J\"ungel and D. Matthes.}
\newblock An algorithmic construction of entropies in
higher-order nonlinear PDEs.
\newblock{\em Nonlinearity}, {\bf 19}:633-659, 2006.

\bibitem{Kar} {E.~A. Karabut.}
\newblock {Two regimes of liquid film flow on a rotating cylinder}.
\newblock {\em J. Appl. Mechanics Technical Phys.},
{\bf 48(1)}:55--64, {2007}.

\bibitem{Kawohl} B. Kawohl.
\newblock Rearrangements and convexity of level sets in PDE.
\newblock Springer Lecture Notes in Mathematics {\bf 1150}, 1985.

\bibitem{Kelmanson} {M. Kelmanson.}
\newblock On inertial effects in the {Moffatt-Pukhnachov} 
coating-flow problem.
\newblock {\em J. Fluid. Mech.} {\bf 633}: 327-353, 2009.

\bibitem{Laug} {R.~S. Laugesen.}
\newblock {New dissipated energies for the thin fluid film equation}.
\newblock {\em Comm. Pure Appl. Anal.}, {\bf 4(3)}:613--634, 2005.

\bibitem{LaugPugh} {R.~S. Laugesen and M.~C.~Pugh.}
\newblock {Properties of steady states for thin film equations}.
\newblock {\em Europ. J. Appl.Math.}, {\bf 11}:293--351, 2000.

\bibitem{LL} {E. H. Lieb and M. Loss.}
\newblock Analysis.
\newblock AMS Graduate Studies in Mathematics {\bf 14}, 1997/2001.

\bibitem{LifSly} {I.~M. Lifshitz and V.~V. Slyozov.}
\newblock {The kinetics of precipitation from supersaturated 
solid solutions},
\newblock {\em J. Physics and Chemistry of Solids},
{\bf 19(1,2)}:35--50, 1961.

\bibitem{MMS} {D. Mattes, R. J. McCann, and G. Savar\'e.}
\newblock A family of nonlinear
fourth order equations of gradient flow type.
\newblock {\em Commun. Partial Differential Equations},
{\bf34}:1352--1397, 2009.


\bibitem{MCC} {T. G. Myers, J. P. F. Charpin, and  S. J. Chapman.}
\newblock The flow and solidification of a thin fluid film
  on an arbitrary three-dimensional surface.
\newblock {\em Physics of Fluids} {\bf 14 (8)}:2788-2803, 2002

\bibitem{Moffatt} {H. K. Moffatt.}
\newblock Behavior of a viscous film on outer surface
of a rotating cylinder.
\newblock {\em J. de M\'ecanique}  {\bf 16(5)}:651-673, 1977.

\bibitem{Otto} {F. Otto.}
\newblock Lubrication approximation with prescribed nonzero contact angle:
an existence result.
\newblock {\em Commun. Partial Differential Equations},
{\bf 23}:2077--2161, 1998.

\bibitem{PSz} {G. P\'olya and G. Szeg\H{o}.}
\newblock Isoperimetric  inequalities in mathematical physics.
\newblock {\em Ann. Math. Stud.} {\bf 27}, Princeton University Press, 1952.

\bibitem{Pukh1} {V.~V. Pukhnachev.}
\newblock Motion of a liquid film on the surface of a rotating cylinder in a
  gravitational field.
\newblock {\em J. Appl. Mechanics and Technical Physics},
{\bf 18(3)}:344--351, 1977.

\bibitem{Pukh3} {V.~V. Pukhnachev.}
\newblock Asymptotic solution of the rotating problem.
\newblock {\em Izv. Vyssh. Uchebn. Zaved. Severo-Kavkaz. Reg. Estestv. Nauk.},
Special Issue on Mathematics and Continuum Mechanics, p. 191--199,
2004.

\bibitem{Tudorascu} {A.~Tudorascu.}
\newblock Lubrication approximation for thin viscous films:
asymptotic behavior of nonnegative solutions.
\newblock {\em Commun. Partial Differential Equations},
{\bf 32}:1147--1172, 2007.


\end{thebibliography}
\end{document}